\documentclass[11pt]{article}

\usepackage{latexsym,amsfonts,amsmath,amssymb,amsthm,url,graphics,psfrag,amscd}
\usepackage[english]{babel}
\usepackage[latin1]{inputenc}
\usepackage[T1]{fontenc}
\usepackage{graphics}
\usepackage{graphicx}
\usepackage{color}

\newcommand{\R}{\mathbb{R}}

\newcommand{\N}{\mathbb{N}}

\newcommand{\Z}{\mathbb{Z}}

\newcommand{\pp}{\mathbb{P}}

\newcommand{\kC}{\mathcal{C}}

\newcommand{\lin}{\left[\kern-0.15em\left[}
\newcommand{\rin} {\right]\kern-0.15em\right]}
\newcommand{\linf}{[\kern-0.15em [}
\newcommand{\rinf} {]\kern-0.15em ]}
\newcommand{\ilin}{\left]\kern-0.15em\left]}
\newcommand{\irin} {\right[\kern-0.15em\right[}

\newtheorem{lem}{Lemma}[section]

\newtheorem{prop}[lem]{Proposition}
\newtheorem{theo}[lem]{Theorem}

\newtheorem{cor}[lem]{Corollary}
\newtheorem {assum}[lem] {Assumption}
\newtheorem {defin}[lem] {Definition}
\newtheorem {rem}[lem] {Remark}

\begin{document}





\title{\textbf{Localization of a vertex reinforced random walk on $\Z$ with sub-linear weight}}

\author{
\normalsize{\textsc{Anne-Laure Basdevant}\footnote{Laboratoire
Modal'X, Université Paris Ouest, France. E-mail:
anne.laure.basdevant@normalesup.org},  \textsc{Bruno
Schapira}\footnote{Département de Mathématiques, Université Paris
XI, France. E-mail: bruno.schapira@math.u-psud.fr} \;and
\textsc{Arvind Singh}\footnote{Département de Mathématiques,
Université Paris XI, France. E-mail: arvind.singh@math.u-psud.fr}}}
\date{}
\maketitle

\vspace*{0.2cm}

\begin{abstract}
We consider a vertex reinforced random walk on the integer lattice
with sub-linear reinforcement. Under some assumptions on the regular
variation of the weight function, we characterize whether the walk
gets stuck on a finite interval. When this happens, we estimate the
size of the localization set. In particular, we show that, for any
odd number $N$ larger than or equal to $5$, there exists a vertex
reinforced random walk which localizes with positive probability on
exactly $N$ consecutive sites.
\end{abstract}

\bigskip
{\small{

\bigskip
\noindent{\bf A.M.S. Classification. } 60K35; 60J17; 60J20

\section{Introduction}

The aim of this paper is to study a vertex reinforced random walk
(VRRW) on the integer lattice $\Z$ with weight sequence $(w(n),n\ge
0) \in (0,\infty)^\N$, that is, a stochastic process $X$ with
transition probabilities given by
\begin{eqnarray*}
\pp\{X_{n+1}=X_n-1 \mid X_0,X_1,\ldots,X_n\}&=&1-\pp\{X_{n+1}=X_n+1 \mid X_0,X_1,\ldots,X_n\}\\
&=&\frac{w(Z_n(X_n-1))}{w(Z_n(X_n-1))+w(Z_n(X_n+1))},
\end{eqnarray*}
where $Z_n(x)$ denotes the number of visits of $X$ to site $x$ up to
time $n$. Assuming that the sequence $w$ is non-decreasing, the walk
has a tendency to favour sites previously visited multiple times
before which justifies the denomination  "reinforced".

This process was introduced by Pemantle \cite{P} in $1992$ and
subsequently studied by several authors (see for instance
\cite{PV,Sch,T1,T2,V2}  as well as Pemantle's survey \cite{PS} and
the references therein). A particularly interesting feature of the
model is that the walk may get stuck on a finite set provided that
the weight sequence $w$ grows sufficiently fast. For instance, in
the linear case $w(n) =n+1$, it was proved in \cite{PV,T1} that the
walk ultimately localizes, almost surely, on five consecutive sites.
Furthermore, if the weight sequence is non-decreasing and grows even
faster (namely $\sum 1/w(n) < \infty$), then the walk localizes
almost surely on two sites \emph{c.f.} \cite{T2}. On the other hand,
if the weight sequence is  regularly varying at infinity with index
strictly smaller than $1$, Volkov \cite{V2} proved that the walk
cannot get stuck on any finite set (see also \cite{Sch} for refined results in this case).

These previous studies left open the critical case where the index
of regular variation of $w$ is equal to $1$ (except for linear
reinforcement). In a recent paper \cite{BSS}, the authors studied
the VRRW with super-linear weights and showed that the walk may
localize on $4$ or $5$ sites depending on a simple criterion on the
weight sequence. In this paper, we consider the remaining case where
the weight function grows sub-linearly. We are interested in finding
whether the walk localizes  and, if so, to estimate the size of the
localization set. More precisely, in the rest of the paper, we will
consider weight sequences which satisfy the following properties:
\begin{assum}\label{assumw}\
\begin{itemize}
\item[(i)] The  sequence $(w(n))_{n\ge 0}$ is positive, non-decreasing, sub-linear and regularly
varying with index $1$ at infinity. Therefore, it can be written in
the form:
$$
w(n) := \frac{n}{\ell(n)}\qquad\hbox{where the sequence $\ell(n)$
satisfies}\, \left\{
\begin{array}{l}
\lim_{n\to\infty}\ell(cn)/\ell(n) = 1 \hbox{ for all
$c>0$}\\
\lim_{n\to\infty}\ell(n) = \infty.
\end{array}
\right.
$$
\item[(ii)] The sequence $\ell(n)$ is eventually non-decreasing.
\end{itemize}
\end{assum}

\begin{rem}
Part (i) of the assumption is quite natural. It states that the
reinforcement is sub-linear yet close enough to linear so that it is
not covered by Volkov's paper \cite{V2}. It would certainly be nice
to relax the assumption of regular variation on $w$ but the
techniques used in this article crucially need it.
 On the contrary, (ii) is of a technical nature and is only required for proving the technical (yet essential) Lemma \ref{WL}. We
believe that it does not play any significant role and that the
results obtained in this paper should also hold without this
assumption.
\end{rem}

It is convenient to extend a weight sequence $w$ into a function so
that we may consider $w(n)$ for non-integer values of $n$. Thus, in
the following, we will call \textit{weight function} any continuous,
non-decreasing function $w:[0,\infty)\to (0,\infty)$. Given a weight
function, we associate the weight sequence obtained by taking its
restriction to the set $\N$ of integers. Conversely, to any weight
sequence $w$, we associate the weight function, still denoted $w$,
obtained by linear interpolation. It is straightforward to check that,
if a sequence $w$ fulfills Assumption \ref{assumw}, then its
associated weight function satisfies
\begin{itemize}
\item[(i)] $w:[0,\infty)\to(0,\infty)$ is a continuous, non-decreasing, sub-linear function which is regularly varying
with index $1$ at infinity. In particular, we can write $w$ in the
form:
$$
w(x) := \frac{x}{\ell(x)}\qquad\hbox{where}\quad \left\{
\begin{array}{l}
\lim_{x\to\infty}\ell(cx)/\ell(x) = 1 \hbox{ for all
$c>0$,}\\
\lim_{x\to\infty}\ell(x) = \infty.
\end{array}
\right.
$$
\item[(ii)] The function $\ell$ is eventually non-decreasing.
\end{itemize}
Therefore, in the rest of the paper, we will say that a weight
function satisfies Assumption \ref{assumw} whenever it fulfills (i)
and (ii) above. In order to state the main results of the paper, we
need to introduce some notation. To a weight function $w$, we
associate $W:[0,\infty)\to[0,\infty)$ defined by
\begin{equation}\label{defW}
W(x) := \int_0^x \frac 1{w(u)}\, du.
\end{equation}
Under Assumption \ref{assumw}, we have $\lim_\infty W = \infty$ so
that $W$ is an increasing homeomorphism on $[0,\infty)$ whose
inverse will be denoted by $W^{-1}$. Consider the operator $G$
which, to each measurable non-negative function $f$ on $\R_+$,
associates the function $G(f)$ defined by
\begin{equation}
\label{opG} G(f)(x):=\int_0^x\frac{w(W^{-1}(f(u))}{w(W^{-1}(u))}\,
du.
\end{equation}
We denote by $G^{(n)}$ the $n$-fold of $G$. For $\eta \in (0,1)$,
define the parameter:
\begin{equation}\label{defiw}
i_\eta(w):=\inf\left\{n\ge 2\; : \; G^{(n-1)}(\eta \hbox{Id}) \
\textrm{ is bounded} \right\},
\end{equation}
where $\hbox{Id}$ stands for the identity function  with the
convention $\inf\emptyset = +\infty$. Since $w$ is non-decreasing,
the map $\eta \mapsto i_\eta(w)$ is also non-decreasing. So we can
define $i_-(w)$ and $i_+(w)$ respectively as the left and right
limits at $1/2$:
\begin{equation}\label{defiplusmoins}
i_\pm(w):=\lim_{\eta\rightarrow {\frac{1}{2}}^\pm} i_\eta(w).
\end{equation}
 As we shall see later, the numbers $i_+(w)$ and $i_-(w)$ are either both
infinite or both finite and in the latter case, we have $i_+(w)-
i_-(w)\in \{0,1\}$. Let us also mention that, although there exist
weight functions for which $i_+(w)\neq i_-(w)$, those cases are
somewhat exceptional and correspond to critical cases for the
asymptotic behaviour of the VRRW (see Remark \ref{ipmdifferent}). We
say that a walk localizes if its range, \emph{i.e.} the set of sites
which are  visited, is finite. Our main theorem about
localization of a VRRW on $\Z$ is the following.

\begin{theo}\label{locps} Let $X$ be a  VRRW on $\Z$ with weight $w$ satisfying Assumption
\ref{assumw}. We have the equivalence
$$
i_{+}(w)<\infty \; \Longleftrightarrow \; i_{-}(w)<\infty \;
\Longleftrightarrow \;  X \mbox{  localizes with positive
probability } \; \Longleftrightarrow \;   X \mbox{  localizes a.s.}
$$
Let $R$ be the random set of sites visited infinitely often by the
walk and denote by $|R|$ its cardinality. When localization occurs
(\emph{i.e.} $i_\pm(w)<\infty$) we have
\begin{eqnarray*}
&(i)&\hbox{$|R| > i_-(w) + 1$ almost surely,}\\
&(ii)&\pp\big\{ 2i_-(w)+1 \le |R|\le 2i_+(w)+1 \big\}>0.
\end{eqnarray*}
\end{theo}
The lower bound on $|R|$  given in ($i$) can be slightly improved
for small values of $i_-(w)$ using a different approach which relies on arguments similar to those introduced by Tarrès in \cite{T1,T2}.
\begin{prop}\label{theoic3} Assume that $w$ satisfies Assumption
\ref{assumw}.
\begin{itemize}
\item[(i)] If $i_-(w) = 2$ then $|R| > 4$ almost surely.
\item[(ii)] If $i_-(w) = 3$ then $|R| > 5$ almost surely.
\end{itemize}
\end{prop}

Let us make some comments. The first part of the theorem identifies
weight functions for which the walk localizes. However, although we
can compute $i_\pm(w)$ for several examples, deciding the finiteness of these
indexes is usually rather challenging. Therefore, it would
be interesting to find a simpler test concerning the operator $G$ to
check whether its iterates $G^{(n)}(\eta \hbox{Id})$ are ultimately
bounded. For instance, does there exist a simple integral test on
$w$ characterizing the behaviour of $G$ ?

The second part of the theorem estimates the size of the
localization interval. According to Proposition \ref{calculi} stated
below, ($i$) shows that there exist walks which localize only on
arbitrarily large subsets but  this lower bound is not sharp as
Proposition \ref{theoic3} shows. In fact, we expect the correct
lower bound to be the one given in ($ii$). More precisely we
conjecture that, when localization occurs,
\begin{equation*}
2i_-(w)+1\;  \le \; |R|\; \le \; 2i_+(w)+1 \qquad \hbox{almost
surely.}
\end{equation*}
In particular, when $i_+(w) = i_-(w)$, the walk should localize a.s.
on exactly $2i_\pm(w) +1$ sites. However, we have no guess as to
whether the cardinality of $R$ may be random when the indexes
$i_\pm(w)$ differ. Let us simply recall that, for super-linear
reinforcement of the form $w(n) \sim n\log\log n$, the walk
localizes on $4$ or $5$ sites so that $|R|$ is indeed random in that
case, \emph{c.f.} \cite{BSS}. Yet, the localization pattern for
super-linear weights is quite specific and may not apply in the
sub-linear case considered here.

 Let us also remark that the trapping of a self-interacting random walk on an arbitrary large
subset of $\Z$ was previously observed by Erschler, T\'oth and
Werner \cite{ETW1,ETW2} who considered a model called \emph{stuck
walks} which mixes both repulsion and attraction mechanisms.
Although stuck walks and VRRWs both localize on large sets, the
asymptotic behaviours of these processes are very different. For
instance, the local time profile of a stuck walk is such that it
spends a positive fraction of time on every site visited infinitely
often. On the contrary, the VRRW exhibits localization patterns
where the walk spends most of its time on three consecutive sites
and only a negligible fraction of time on the other sites of $R$
(\emph{c.f.} Section \ref{sectionlocaltime} for a more detailed
discussion on this subject).

\medskip

As we already mentioned, we can compute $i_\pm(w)$ for particular
classes of weight functions. The case where the slowly varying
function $\ell(x)$ is of order $\exp(\log^\alpha(x))$ turns out to
be particularly interesting.

\begin{prop}\label{calculi} Let $w$ be a non-decreasing weight sequence such that
$$
w(k)\; \underset{k\to\infty}{\sim} \; \frac{k}{\exp(\log^\alpha
k)}\qquad\hbox{for some $\alpha\in(0,1)$.}
$$
Then $i_-(w)$ and $i_+(w)$ are both finite.
 Moreover, for
$n\in\N^*$, we have
$$ \alpha \in \left(\frac{n-1}{n},\frac{n}{n+1}\right) \quad \Longrightarrow \quad i_-(w)=i_+(w)=n+1.$$
\end{prop}

The proposition implies that, for any odd number $N$ larger than or equal
to $5$, there exists a VRRW which localizes on exactly $N$ sites
with positive probability. It is also known from previous results
\cite{BSS,V2} that a VRRW may localize on $2$ or $4$ sites (but it
cannot localize on $3$ sites). We wonder whether there exist any
other admissible values for $|R|$ apart from $2,4,5,7,9,\ldots$ Let
us also mention that, using monotonicity properties of $i_\pm$, it
is possible to construct a weight function $w$, regularly varying
with index $1$, which is growing slower than
$x/\exp(\log^\alpha(x))$ for any $\alpha <1$ such that $i_\pm(w)
=\infty$. For example, this is the case if
$$w(x) \sim \frac{x}{\exp\left(\frac{\log x}{\log\log x}\right)}$$
\emph{c.f.} Corollary \ref{ipminfini}. Hence, a walk with such
reinforcement does not localize. However, we expect it to have a
very unusual behaviour: we conjecture it is recurrent on $\Z$ but
spends asymptotically all of its time on only three sites.

\medskip

Let us give a quick overview of the strategy for the proof of
Theorem \ref{locps}. The main part consists in establishing a similar
result for a reflected VRRW $\bar{X}$ on the half-line $\lin
-1,+\infty\irin$. In order to do so, we introduce two alternative
self-interacting random walks $\widetilde{X}$ and $\widehat{X}$
which, in a way, surround the reflected walk $\bar{X}$. The
transition mechanisms of these two walks are such that, at each time
step, they jump to their left neighbour with a probability
proportional to a function of the site local time on their left,
whereas they jump to the right with a probability proportional to a
function of the edge local time on their right. It is well known
that an \emph{edge} reinforced random walk on $\Z$ (more generally, on any acyclic graph) may be
constructed from a sequence of i.i.d. urn processes, see for instance Pemantle \cite{PE}. Subsequently, in the case of \emph{vertex} reinforced random walks,
Tarr\`es \cite{T1} introduced martingales attached to each site, which play a similar role as urns, but
a major difficulty is that they are, in that case, strongly correlated.
Considering walks
$\widetilde{X},\widehat{X}$ with a mixed \emph{site}/\emph{edge}
reinforcement somehow gives the best of both worlds: it enables to
simplify the study of these walks by creating additional structural
independence (in one direction) while still preserving the flavor
and complexity of the site reinforcement scheme. In particular,
$\widetilde{X}$,$\widehat{X}$ have the nice restriction property
that their laws on a finite set do not depend upon the path taken by
the walks on the right of this set. Considering reflected walks, we
can then work by induction and prove that when the critical indexes
$i_\pm$ are finite, $\widetilde{X}$,$\widehat{X}$ localize on
roughly $i_{\pm} +1$ sites. Then, in turn, using a coupling argument
we deduce a similar criterion for the reflected VRRW $\bar{X}$. The
last step consists in transferring these results to the
non-reflected VRRW on $\Z$. The key point here being that the
localization pattern for $\widetilde{X}$,$\widehat{X}$ has a
particular shape where the urn located at the origin is balanced,
\emph{i.e.} sites $1$ and $-1$ are visited about half as many times
as the origin. This fact permits to use symmetry arguments to
construct a localization pattern for the non reflected walk of size
of order $2 i_{\pm} +1$.

\medskip

The rest of the paper is organized as follows. In Section
\ref{sectiontwo}, we prove Proposition \ref{calculi} and collect
several results concerning the critical indexes which we will need
later on during the proof of the theorem. In Section \ref{seccouplage},
we introduce the three walks $\widetilde{X}$,$\widehat{X}$ and
$\bar{X}$ mentioned above and we prove coupling properties between
these processes. Sections \ref{sectiontilde} and \ref{sectionhat}
are respectively devoted to studying the walks $\widetilde{X}$ and
$\widehat{X}$. In Section \ref{sectionbar}, we rely on the results
obtained in the previous sections to describe the asymptotic
behaviour of $\bar{X}$. The proof of Theorem \ref{locps} is then
carried out in Section \ref{sectionmaintheo} and followed in Section
\ref{sectionlocaltime} by a discussion concerning the shape of the
asymptotic local time profile. Finally we provide in the appendix a
proof of Proposition \ref{theoic3} which, as we already mentioned,
uses fairly different technics but is still included here for the
sake of completeness.

\section{Preliminaries: properties of $W$ and $i_\pm(w)$}\label{sectiontwo}

The purpose of this section is to study the operator $G$ and collect
 technical results from real analysis concerning regularly varying
functions. As such, this section is not directly related with VRRW
and does not involve probability theory. The reader interested in
the main arguments used for proving Theorem \ref{locps} may wish to
continue directly to Section $3$ after simply reading the statement
of the results of this section.

\subsection{Some properties of the slowly varying function $W$}

From now on, we assume that all the weight functions considered
satisfy Assumption \ref{assumw} (i).

\begin{lem}\label{limborne}
The function $W$ defined by \eqref{defW} is slowly varying
\emph{i.e.}
$$
W(cx)\underset{x\to\infty}{\sim} W(x) \qquad\hbox{ for any $c>0$}.
$$
Moreover, given two positive functions $f$ and $g$ with $\lim_{\infty} f=\lim_{\infty} g=+\infty$, we
have
\begin{eqnarray}
\label{limborne_eq1}\limsup_{x\to\infty}\frac{W(f(x))}{W(g(x))} <1 & \Longrightarrow &  \lim_{x\to\infty} \frac{f(x)}{g(x)}= 0,\\
\label{limborne_eq2} \sup_{x\geq 0} \big( W(f(x))-W(g(x)) \big)
<\infty & \Longrightarrow & \limsup_{x\to\infty} \frac{f(x)}{g(x)}
\leq 1,\\
\label{limborne_eq3} \sup_{x\geq 0} \big| W(f(x))-W(g(x)) \big|
<\infty & \Longrightarrow & \lim_{x\to\infty} \frac{f(x)}{g(x)} = 1.
\end{eqnarray}
\end{lem}
\begin{proof} The fact that $W$ is slowly varying follows from
Proposition $1.5.9a$ of \cite{BGT}. Assume now that $\limsup f/g >\lambda >0$. Then, there exists an increasing sequence $(x_n)$ such
that
$$\limsup_{x\to\infty}\frac{W(f(x))}{W(g(x))}\ \ge\ \lim_{n\to\infty}\frac{W(\lambda g(x_n))}{W(g(x_n))}\, =\, 1.$$
which proves \eqref{limborne_eq1}. Concerning the second assertion,
the uniform convergence theorem for regularly varying functions shows
that, for $\lambda > 0$ (\emph{c.f.} \cite{BGT} p.$127$ for details),
$$\lim_{x \to \infty}\, \frac{W(\lambda x)-W(x)}{\ell(x)}=\log \lambda,$$
where $\ell$ is the slowly varying function associated with $w$.
Therefore, if $\limsup f/g
> \lambda
> 1$, there exist arbitrarily large $x$'s such that
$$
W(f(x))-W(g(x))\ge W(\lambda g(x))-W(g(x))\ge
\frac{1}{2}\log(\lambda)\ell(g(x)),
$$
which implies that $W(f(\cdot))-W(g(\cdot))$ is unbounded from
above. Finally, Assertion \eqref{limborne_eq3} follows from
\eqref{limborne_eq2} by symmetry.
\end{proof}

Given a measurable, non negative function $\psi :\R_+\to \R_+$, we
introduce the notation $W_\psi$ to denote the function
\begin{equation}\label{defWf}
W_\psi(x):=\int_0^x \frac{du}{w(u+\psi(u))}.
\end{equation}
In the linear case $\psi(u)=\eta u$ with $\eta>0$, we shall simply write
$W_\eta$ instead of $W_\psi$ (note that $W_0 = W$). The next result
is a slight refinement of \eqref{limborne_eq3}.
\begin{lem}\label{limWWchapeau}
Let $\psi$ be a measurable non-negative function such that
$$W(x)-W_{\psi}(x)=o(\ell(x)) \quad\hbox{as $x\to\infty$.}$$
Then, for any positive functions $f$ and $g$ with $\lim_{\infty} f=\lim_{\infty} g=+\infty$, we have
$$\sup_{x\geq 0}\big| W(f(x))-W_{\psi}(g(x))\big| < \infty \ \Longrightarrow \ \lim_{x\to\infty}\, \frac{f(x)}{g(x)}= 1.$$
\end{lem}
\begin{proof} Since $\psi$ is non-negative, we have $W_\psi\le W$ thus Lemma \ref{limborne} yields $\limsup  f/g\leq
1$. Fix $0<\lambda<1$. We can write
\begin{equation*}
W(\lambda g(x))-W(f(x)) = W(\lambda g(x))-W(g(x)) +
W(g(x))-W_\psi(g(x)) + W_\psi(g(x))-W(f(x)).
\end{equation*}
Using the facts that $$W(x)-W_{\psi}(x)=o(\ell(x))\qquad \mbox{ and
} \qquad W(\lambda x)-W(x)\sim \log(\lambda)\ell(x),$$ we deduce
that, if $W_\psi(g(\cdot))-W(f(\cdot))$ is bounded from above, then
$W(\lambda g(\cdot))-W(f(\cdot))$ is also bounded from above. In
view of Lemma \ref{limborne}, this yields $\limsup g/f \leq
1/\lambda$ and we conclude the proof of the lemma letting $\lambda$
tend to $1$.
\end{proof}

\noindent We conclude this subsection by showing that the function
\begin{equation}\label{defphi2}
\Phi_{\eta,2}(x):= W^{-1}(\eta W(x/\eta))
\end{equation}
satisfies the hypothesis of the previous lemma for any $\eta\in
(0,1)$. As we have already mentioned in the introduction, the following
lemma is the only place in the paper where we require $\ell$ to be
eventually non-decreasing.

\begin{lem}\label{WL}
Assume that $w$ also satisfies (ii) of Assumption \ref{assumw}. Let
$\eta\in (0,1)$, we have
\begin{equation}\label{WLeq}
W(x)-W_{\Phi_{\eta,2}}(x)=o(\ell(x))  \quad\hbox{as $x\to\infty$.}
\end{equation}
Furthermore, there exists a non-decreasing function
$f_\eta:[0,\infty)\to [0,\infty)$ such that
\begin{eqnarray*}
&(a)& f_\eta \geq \Phi_{\eta,2}\\
&(b)& f_\eta = o(x) \\
&(c)& W(x) - W_{f_\eta}(x) = o(\ell(x))\\
&(d)& \lim_{x\to+\infty} W(x) - W_{f_\eta}(x) = +\infty.
\end{eqnarray*}
\end{lem}
\begin{proof} Choose $x_0$ large enough such that $\ell$ is non-decreasing on $[x_0,\infty)$.
Let $C:= W(x_0)-W_{\Phi_{\eta,2}}(x_0)$. For $x\ge x_0$, we get
\begin{eqnarray*}
W(x)-W_{\Phi_{\eta,2}}(x) &=& C+\int_{x_0}^x
\left(\frac{\ell(u)}{u}-
\frac{\ell(u+\Phi_{\eta,2}(u))}{u+\Phi_{\eta,2}(u)}\right)\, du\\
&\le& C+ \int_{x_0}^x \frac{\ell(u)\Phi_{\eta,2}(u)}{u^2} \, du \\
&=& C+ \int_{x_0}^x \frac{W^{-1}(\eta W(u/\eta))}{ w(u) u}\, du\\
&\leq& C' + \frac{2}{\eta} \int_{x_0}^x \frac{W^{-1}(\eta
W(u/\eta))}{ w(u/\eta) u}\,du,
\end{eqnarray*}
where we used  $\eta w(u/\eta)\sim w(u)$ as $u\to\infty$ and
where $C'$ is a finite constant. From the change of variable $t =
W(u/\eta)$, it follows that
\begin{equation*}
W(x)-W_{\Phi_{\eta,2}}(x)\le C' +
\frac{2}{\eta}\int_{W(x_0/\eta)}^{W(x/\eta)}\frac{W^{-1}(\eta
t)}{W^{-1}(t)}\,dt.
\end{equation*}
Now let
$$
J_\eta(x) := \int_0^x \frac{W^{-1}(\eta u)}{W^{-1}(u)}\, du,
$$
which is well-defined since $\lim_{u\to 0} W^{-1}(\eta
u)/W^{-1}(u) = \eta$. It remains to prove that
\begin{equation}\label{techJb}
J_\eta(x) = o(\ell(W^{-1}(x)))\quad\hbox{when $x\to\infty$,}
\end{equation}
as this will entail
\begin{equation*}
W(x)-W_{\Phi_{\eta,2}}(x) \leq C' + \frac{2}{\eta}J_\eta(W(x/\eta))
 = o(\ell(x/\eta)) = o(\ell(x)).
\end{equation*}
In order to establish \eqref{techJb}, we consider the function $h(x)
:= \log W^{-1}(x) $. This function is non-decreasing and
$$h'(x)=\frac{w(W^{-1}(x))}{W^{-1}(x)}=\frac{1}{\ell(W^{-1}(x))}.$$
Thus, we need to prove that
$$\lim_{x\to\infty}\, h'(x) J_\eta(x)\, =\, \lim_{x\rightarrow \infty}\, h'(x)\int_0^x e^{h(\eta u)-h(u)}\, du\, =\, 0.$$
Using that $h'$ is non-increasing, we get, for any $A\in [0,z]$,
\begin{eqnarray*}
J_\eta(x)&\le& \int_0^x e^{-(1-\eta)uh'(u)}\, du \\
&\le & \int_0^A e^{-(1-\eta)uh'(A)}\, du+ \int_A^\infty e^{-(1-\eta)uh'(x)}\, du \\
&=& \frac{1}{(1-\eta)h'(A)}+
\frac{e^{-(1-\eta)Ah'(x)}}{(1-\eta)h'(x)}.
\end{eqnarray*}
According to Equation $1.5.8$ of \cite{BGT} p.$27$, we have
$\ell(x)=o(W(x))$ hence
$$1/h'(x)=\ell(W^{-1}(x))=o(x)\quad\hbox{as $x\to\infty$.}$$
Fix $\varepsilon>0$ and set $A:= A(x) =
1/(\sqrt{\varepsilon}h'(x))$. Then, for all $x$ large enough such
that $1/h'(A)\le \varepsilon A$, we get
$$(1-\eta)h'(x) J_\eta(x)\le \sqrt{\varepsilon}+e^{-(1-\eta)/\sqrt{\varepsilon}},$$
which completes the proof of \eqref{WLeq}.

Concerning the second part of the lemma, it follows from Lemma
\ref{limborne} that $\Phi_{\eta,2}(x) = o(x)$ for any $0<\eta<1$ (see
also Lemma \ref{lemmtech}). Hence, if $\lim_\infty W
-W_{\Phi_{\eta,2}} = \infty$, then we can simply choose $f_\eta =
\Phi_{\eta,2}$. Otherwise, we can always construct a positive
non-decreasing function $h$ such that $f_\eta := \Phi_{\eta,2} + h$
is a solution (for instance, one can construct $h$ continuous with
$h(0)=0$, piecewise linear, flat on intervals $[x_{2n},x_{2n+1}]$
and with slope $1/n$ on the intervals $[x_{2n+1},x_{2n+2}]$ where
$(x_i)_{i\geq0}$ is a suitably chosen increasing sequence). The
technical details are left to the reader.
\end{proof}

\subsection{Properties of the indexes $i_\pm(w)$}

Recall the construction of the family $(i_\eta(w),\eta\in(0,1) )$
from the operator $G$ defined in \eqref{opG}. In this subsection, we
collect some useful results concerning this family. We show in
particular that the map $\eta\mapsto i_\eta(w)$ can take at most two
different (consecutive) values. In order to do so, we provide an
alternative description of these parameters in term of another
family $(j_\eta,\eta\in(0,1) )$ defined using another operator $h$
whose probability interpretation will become clear in the next
sections. More precisely, let $H$ be the operator which, to each
homeomorphism $f: [0,\infty) \to [0,\infty)$, associates the
function $H(f): [0,\infty) \to [0,\infty)$ defined by
\begin{equation}\label{defH}
H(f)(x):=W^{-1}\left(\int_0^x\frac{du}{w(f^{-1}(u))}\right)\qquad\hbox{for
$x\geq 0$,} \end{equation} where $f^{-1}$ stands for the inverse of
$f$. If $H(f)$ is unbounded, then it is itself an homeomorphism.
Thus, for each $\eta\in(0,1)$, we can define by induction the
(possibly finite) sequence of functions $(\Phi_{\eta,j}, 1 \leq
j\leq j_\eta(w) )$ by
\begin{equation}\label{defphi}
\left\{
\begin{array}{ll}
\Phi_{\eta,1}:= \eta\hbox{Id}& \\
\Phi_{\eta,j+1}:= H(\Phi_{\eta,j}) & \hbox{ if $\Phi_{\eta,j}$ is
unbounded,}
\end{array}\right.
\end{equation}
where
\begin{equation*}
j_\eta(w):=\inf\{j\ge 1\; :\; \Phi_{\eta,j}\ \textrm{ is bounded}\}.
\end{equation*}
We use the convention $\Phi_{\eta,j} = 0$ for $j > j_\eta(w)$. Let us
remark that this definition of $\Phi_{\eta,2}$ coincides with the
previous definition given in \eqref{defphi2}. In particular,
$\Phi_{\eta,2}$ is always unbounded, which implies
$$j_\eta(w) \in \lin 3,+\infty\rin.$$
\begin{lem}
\label{propH} The operator $H$ is monotone in the following sense:
\begin{itemize}
\item[\textup{(i)}] If $f\le g$, then $H(f) \le H(g)$.
\item[\textup{(ii)}] If $f(x)\le g(x)$, for all $x$ large enough and  $H(f)$ is unbounded,
then $\limsup H(f)/H(g) \le 1$.
\end{itemize}
\end{lem}
The proof of the lemma is straightforward so we omit it. The
following technical results will be used in many places throughout
the paper.

\begin{lem}
\label{lemmtech} Let $0<\eta<\eta'<1$ and $\lambda>0$. For all
$j\in \lin 2, j_\eta(w)-1\rin$, we have, as $x\to\infty$,
\begin{eqnarray*}
\textup{(i)} &\!&\!\!\!\! \Phi_{\eta,j}(x)=o(x),\\
\textup{(ii)}&\!&\!\!\!\! \Phi_{\eta,j}(\lambda x)=
o(\Phi_{\eta',j}(x)).
\end{eqnarray*}
\end{lem}
\begin{proof} As we already mentioned, we have $W(\Phi_{\eta,2}(x))=\eta
W(x/\eta)$ hence Lemma \ref{limborne} implies that
$\Phi_{\eta,2}(x)=o(x)$ and (i) follows from Lemma \ref{propH}. We
prove (ii) by induction on $j$. Recalling that $W$ is slowly
varying, we have
\begin{equation*}
\limsup_{x\to\infty} \frac{W(\Phi_{\eta,2}(\lambda
x))}{W(\Phi_{\eta',2}(x))}=\frac{\eta}{\eta'}\, \limsup_{x\to\infty}
\frac{W(\lambda x/\eta)}{W(x/\eta')}=\frac{\eta}{\eta'}<1,
\end{equation*}
 which, by using Lemma \ref{limborne}, yields $\Phi_{\eta,2}(\lambda x)=o(\Phi_{\eta',2}(x))$.
Let us now assume that for some $j < j_\eta(w)-1$,
$\Phi_{\eta,j}(x)=o(\Phi_{\eta',j}(x/\lambda))$ for all $\lambda
>0$. Fix $\delta>0$. Using again the monotonicity
property of $H$, we deduce that
\begin{equation*}
\limsup_{x\to\infty} \frac{\Phi_{\eta,j+1}(x)}{H\left(\delta
\Phi_{\eta',j}\left(\frac{\cdot}{\lambda}\right)\right)(x)} \leq 1.
\end{equation*}
Notice that
\begin{equation*}
H\left(\delta\Phi_{\eta',j}\left(\frac{\cdot}{\lambda}\right)\right)(x)
\ = \ W^{-1}\left( \int_0^{x/\delta}\frac{\delta
dt}{w(\lambda(\Phi_{\eta',j}^{-1}(t)))}\right) \ \le \ W^{-1}\left(
C+\frac{2\delta}{\lambda}W\left(\Phi_{\eta',j+1}\left(\frac{x}{\delta}\right)\right)\right),
\end{equation*}
where we used that $w(\lambda x)\le 2\lambda w(x)$, for $x$ large
enough and where $C$ is some positive constant. Moreover, Lemma
\ref{limborne} shows that, for $C>0$, $\varepsilon \in (0,1)$ and
any positive unbounded function $f$, we have
\begin{equation*}
W^{-1}(\varepsilon f(x)+C)=o(W^{-1}( f(x)).
\end{equation*}
Hence, choosing $\lambda$ such that $2\delta< \lambda$, we find that
$$H\left(\delta\Phi_{\eta',j}\left(\frac{\cdot}{\lambda}\right)\right)(x)=o\left(
\Phi_{\eta',j+1}(\frac{x}{\delta})\right),$$ which concludes the
proof of the lemma.
\end{proof}

\noindent We can now prove the main result of this section which
relates $j_\eta(w)$ and $i_\eta(w)$.

\begin{prop}\label{prop_iegalj} The maps $\eta\mapsto i_{\eta}(w)$ and $\eta\mapsto
j_{\eta}(w)$ are non-decreasing and take at most two consecutive
values. Moreover, at each continuity point $\eta$ of $j_\eta(w)$, we
have
\begin{equation}\label{iegalj}
j_{\eta}(w)=i_{\eta}(w)+1.
\end{equation}
\end{prop}
\begin{proof} It is clear that the monotonicity result of Lemma \ref{propH} also holds for the operator $G$ defined
by \eqref{opG}. Thus, both functions $\eta\mapsto j_\eta(w)$ and
$\eta\mapsto i_\eta(w)$ are non-decreasing.  Moreover, according to
(i) of the previous lemma, we have $\Phi_{\eta',2}=o(\Phi_{\eta,1})$
for any $\eta,\eta' \in (0,1)$. Combining (ii) of Lemma \ref{propH}
with (ii) of the previous lemma, we deduce that
$\Phi_{\eta',3}=o(\Phi_{\eta,2})$ for any $\eta,\eta' \in (0,1)$.
Repeating this argument, we conclude by induction that $j_{\eta'}(w)
\leq j_{\eta}(w) +1$ which proves that $\eta\mapsto j_{\eta}(w)$
takes at most two different values. The same property will also hold
for $i_\eta(w)$ as soon as we establish \eqref{iegalj}.

Define $\varphi_{\eta,j}:=W\circ \Phi_{\eta,j}\circ W^{-1}$. Using
the change of variable $z=W(u)$ in \eqref{defH}, we find that, for
$j< j_\eta(w)$,
\begin{equation}\label{defvarphi}
\varphi_{\eta,j+1}(x)=\int_0^x \frac{w\circ W^{-1}(z)}{w\circ W^{-1}
\circ \varphi_{\eta,j}^{-1}(z)}\, dz.
 \end{equation}
Define by induction
\begin{equation*}
\left\{
\begin{array}{ll}
h_{\eta,1} := \varphi_{\eta,1},& \\
h_{\eta,j+1}:=\varphi_{\eta,j+1} \circ h_{\eta,j} & \hbox{for
$j\ge 1$.}\\
\end{array}
\right.
\end{equation*}
We have $h_{\eta,j}=W\circ \Phi_{\eta,j}\circ\ldots \circ
\Phi_{\eta,1}\circ W^{-1}$ thus
$$j_\eta(w)=\inf\{j\ge 3\; :\; h_{\eta,j}\ \hbox{ is bounded}\}.$$
Note that $h_{\eta,2}(x)=\Phi_{\eta,1}(x)=\eta x$. Furthermore,
using the change of variable $z=h_{\eta,j}(u)$ in \eqref{defvarphi},
it follows by induction that, for $j < j_\eta(w)$,
\begin{eqnarray}
\label{hbar} h_{\eta,j+1}(x)&=& \eta\, \int_0^x \frac{w\circ
W^{-1}\circ h_{\eta,j}(u)}{w(\eta W^{-1}(u))}\, du.
\end{eqnarray}
Define also the sequence $(g_{\eta,j})_{j\ge 1}$, by
\begin{equation*}
g_{\eta,j}:=G^{(j-1)}(\Phi_{\eta,1}).
\end{equation*}
Recall that, by definition,
$$i_\eta(w)=\inf\{j\ge 2 \; :\; g_{\eta,j} \ \mbox{ is bounded}\}. $$
Using Lemma \ref{limborne}, it now follows by induction from
\eqref{opG} and \eqref{hbar} that for $\alpha<\eta<\beta$ and $j\ge
2$,
\begin{eqnarray}
\label{hbarh}
\left\{
\begin{array}{ll}
g_{\alpha,j}(x)=o(h_{\eta,j+1}(x)) & \mbox{ as long as $h_{\eta,j+1}$ is unbounded,}\\
h_{\eta,j+1}(x)=o(g_{\beta,j}(x)) & \mbox{ as long as $g_{\beta,j}$
is unbounded.}
\end{array}
\right.
\end{eqnarray}
Therefore,
$$i_\alpha(w)+1\le j_\eta(w) \le i_\beta(w)+1\qquad \mbox{ for all $\alpha<\eta<\beta$},$$
which proves that $j_\eta(w)=i_\eta(w)+1$ if the map $j_\eta(w)$ is
continuous at point $\eta$.
\end{proof}

\subsection{Proof of  Proposition
\ref{calculi}}

For $\eta \in (0,1)$, define
 $$i_{\eta,\pm}(w):=\lim_{\delta \to \eta^\pm} i_\delta(w).$$
In accordance with \eqref{defiplusmoins}, we have $i_{\pm}(w) =
i_{1/2,\pm}(w)$. Given another weight function $\tilde{w}$, we will
use the notation $\tilde{W},\tilde{\Phi},\ldots$ to denote the
quantities $W,\Phi,\ldots$ constructed from $\tilde{w}$ instead of
$w$. The following result compares the critical indexes
$i_{\eta,\pm}$ of two weight functions.
\begin{prop}\label{wequiv}
Let $w,\tilde w$ denote two weight functions and let $\eta \in
(0,1)$.
\begin{itemize}
\item[(i)]  If $w(x)\sim \tilde w(x)$, then  $i_{\eta,\pm}(w)=
i_{\eta,\pm}(\tilde w)$.
\item[(ii)] If the function $(w\circ W^{-1})/(\tilde{w}\circ \tilde W^{-1})$ is eventually non-decreasing, then
$i_{\eta,\pm}(w) \le i_{\eta,\pm}(\tilde w)$.
\end{itemize}
\end{prop}
\begin{proof} Let us first establish (i). We prove by induction on $j$ that, for all $\beta\in (\eta,1)$ and $x$ large enough,
\begin{equation}\label{eqhr}
\Phi_{\eta,j}(x)\le \tilde{\Phi}_{\beta,j}(x)\qquad \mbox{for any
$j< j_{\eta,+}(\tilde{w})$}.
\end{equation}
The assumption that $w(x)\sim \tilde w(x)$ implies that, for all
$\varepsilon>0$ and for $x$ large enough,
\begin{equation*}
\frac{1-\varepsilon}{\tilde w(x)} \le \frac{1}{ w(x)}\le
\frac{1+\varepsilon}{\tilde w(x)} \quad\hbox{and}\quad W^{-1}(x)\le
\tilde{W}^{-1}((1+\varepsilon)x).
\end{equation*}
Assume now that \eqref{eqhr} holds for some
$j<j_{\eta,+}(\tilde{w})-1$ and all $\beta>\eta$. Then, for $x$
large enough
$$\frac{1}{w(\Phi^{-1}_{\eta,j}(x))}\le \frac{1+\varepsilon}{\tilde{w}(\tilde{\Phi}^{-1}_{\beta,j}(x))},$$
which yields, for $x$ large enough,
$$\Phi_{\eta,j+1}(x)=W^{-1}\left(\int_{0}^x \frac{dt}{w(\Phi^{-1}_{\eta,j}(t))}\right)\le
\tilde{W}^{-1}\left((1+\varepsilon)^2\int_{0}^x
\frac{dt}{\tilde{w}(\tilde{\Phi}^{-1}_{\beta,j}(t))}  + C\right),$$
for some constant $C >0$. On the other hand, thanks to Lemma
\ref{lemmtech}, setting $\beta':=(1+\varepsilon)^3\beta$, we have,
$$\tilde{\Phi}^{-1}_{\beta,j}(x)\ge(1+\varepsilon)^3 \tilde{\Phi}^{-1}_{\beta',j}(x).$$
The regular variation of $\tilde{w}$ now implies,
$$(1+\varepsilon)^2\int_{0}^x \frac{dt}{\tilde{w}(\tilde{\Phi}^{-1}_{\beta,j}(t))} + C  \le \int_{0}^x \frac{dt}{\tilde{w}(\tilde{\Phi}^{-1}_{\beta',j}(t))}$$
(where we used the divergence at infinity of the integral on the r.h.s.)
and therefore, for $x$ large enough,
$$\Phi_{\eta,j+1}(x)\le \tilde{\Phi}_{\beta',j+1}(x).$$
This proves \eqref{eqhr} by taking $\varepsilon$ small enough.
Applying $\eqref{eqhr}$ with $j=j_{\eta,+}(\tilde{w})-1$ and
$\beta>\eta$ such that $j_{\eta,+}(\tilde{w})=j_{\beta}(\tilde{w})$,
we get, with similar arguments as before,
$$\Phi_{\eta,j_{\eta,+}(\tilde{w})}(x)\le \tilde{W}^{-1}
\left((1+\varepsilon)^2\int_{0}^\infty
\frac{dt}{\tilde{w}(\tilde{\Phi}^{-1}_{\beta,j_{\eta,+}(\tilde{w})-1}(t))}+C\right)<\infty,$$
which implies $j_{\eta}(w)\le j_{\eta,+}(\tilde{w})$ and therefore
$j_{\eta,+}(w)\le j_{\eta,+}(\tilde{w})$. By symmetry, it follows
that $j_{\eta,+}(w) = j_{\eta,+}(\tilde{w})$. The same result also
holds for $j_{\eta,-}$ using similar arguments. This completes the
proof of (i).

We now prove (ii). To this end, we show by induction on $n$ that,
for any $\eta<\eta'$, $n< i_{\eta'}(\tilde w)$ and $x$ large
enough:
\begin{equation}\label{hypG} G^{(n-1)}(\Phi_{\eta,1})(x)\le
\tilde{G}^{(n-1)}(\Phi_{\eta',1})(x),
\end{equation}
which, in view of \eqref{defiw} will imply $i_\eta(w) \le
i_{\eta'}(\tilde w)$ and therefore $i_{\eta,\pm}(w) \le
i_{\eta,\pm}(\tilde w)$. It is easy to check that
\begin{eqnarray*}
G^{(n-1)}(\Phi_{\eta,1})(x)&\le& x \\
\tilde{G}^{(n-1)}(\Phi_{\eta,1})(x)&=&o(\tilde{G}^{(n-1)}(\Phi_{\eta',1})(x))\quad
 \mbox{ for $\eta<\eta'$ and $n<i_{\eta'}(w)$}.
\end{eqnarray*}
Thus, assuming that \eqref{hypG} holds for some $n <
i_{\eta'}(\tilde w)-1$, we find that, for $x$ large,
\begin{equation*}
\frac{w\circ W^{-1}(G^{(n-1)}(\Phi_{\eta,1})(x))}{w\circ
W^{-1}(x)}\le \frac{\tilde w\circ \tilde
W^{-1}(G^{(n-1)}(\Phi_{\eta,1})(x))}{\tilde w\circ \tilde
W^{-1}(x)}\le \frac{\tilde w\circ \tilde
W^{-1}(\tilde{G}^{(n-1)}(\Phi_{\eta',1})(x))}{\tilde w\circ \tilde
W^{-1}(x)}.
\end{equation*}
By integrating, we get, for any $\eta''>\eta'$,
\begin{equation*}
 G^{(n)}(\Phi_{\eta,1})(x)\ \le \  \tilde
 G^{(n)}(\Phi_{\eta',1})(x)+C \ \le \
 \tilde{G}^{(n)}(\Phi_{\eta'',1})(x),
\end{equation*}
which shows that \eqref{hypG} holds for $n+1$, as wanted. \vspace*{0.3cm}
\end{proof}
We now have all the tools needed for proving Proposition
\ref{calculi} which provides examples of weight sequences $w$ with
arbitrarily large critical indexes.
\begin{proof}[Proof of Proposition \ref{calculi}]
Fix  $\alpha\in (0,1)$ and consider a weight function $w$ such that
\begin{equation}\label{defwalpha}
w(x):=x \exp(-(\log x)^\alpha) \qquad \mbox{  for $x\ge e$.}
\end{equation}
An integration by part yields, for any $\gamma\in(0,1)$ and $x$ large enough
$$\gamma V(x)\le W(x)\le V(x)\quad
\mbox{ where }\quad
V(x):=\frac{1}{\alpha} (\log
x)^{1-\alpha} \exp((\log x)^\alpha). $$
Set $\beta:=1/\alpha$ and define  for $\delta>0$,
$$U_\delta(x)=\exp\left((\log x -(\beta-1)\log \log x +\log{ \alpha\delta} )^{\beta}\right).$$
It is easily checked that, for $x$ large enough, $V\circ U_1(x)\le x$ and $V\circ U_\delta(x)\sim \delta x$. This implies that, for $x$ large enough,
$$U_1(x)\le W^{-1}(x)\le U_2(x).$$
 Let $\eta\in(0,1)$ and define the sequence
of functions $(g_{\eta,k})_{k\ge 1}$ by
$$g_{\eta,k} :=G^{(k-1)}(\eta \hbox{Id}),$$
where $G$ is the operator defined by \eqref{opG}. We prove by induction that, if
$k\ge 1$ is such that $(k-1)(\beta-1)<1$, then there exist two
positive constants $c_1$ and $c_2$ (depending on $k$ and $\eta$),
such that, for $x$ large enough,
\begin{equation}\label{encadrementhn}
x\exp(-c_1(\log x)^{(k-1)(\beta-1)})\le g_{\eta,k}(x) \le
x\exp(-c_2(\log x)^{(k-1)(\beta-1)}),
\end{equation}
and that if $(k-1)(\beta-1)>1$, then $g_{\eta,k}$ is bounded. This
result holds for $k=1$. Assume now that \eqref{encadrementhn} holds
for some $k$ such that $(k-1)(\beta-1)<1$. We have, for $x$ large,
\begin{eqnarray*}
\log \left(\frac{ w\circ  W^{-1}\circ g_{\eta,k}(x)}{
w\circ  W^{-1}(x)}\right) &\le & \log \left(\frac{ w\circ  U_2\circ g_{\eta,k}(x)}{
w\circ  U_1(x)}\right)\\
&=&
\log\left(2\left(\frac{\log g_{\eta,k}(x)}{\log x}\right)^{\beta-1}
\frac{x}{g_{\eta,k}(x)}\right)+\log(U_2\circ g_{\eta,k}(x))-\log U_1(x)\\
&\le& c_1(\log x)^{(k-1)(\beta-1)}+(\log x-c_2(\log x)^{(k-1)(\beta-1)})^{\beta}-(\log x-\beta\log \log x)^\beta\\
&\le & \frac{-\beta c_2}{2}(\log
x)^{(k-1)(\beta-1)-1+\beta}:=-c_2'(\log x)^{\gamma},
\end{eqnarray*}
with $\gamma:=k(\beta-1)$. On the one hand, if $\gamma>1$, then
$g_{\eta,k+1}$ is bounded. On the other hand, if $\gamma<1$, an
integration by part yields
$$\int_0^x \exp(-c_2'(\log u)^{\gamma})du\sim x\exp(-c_2'(\log x)^\gamma),$$
giving the desired upper bound for $g_{\eta,k+1}$  (if $\gamma=1$, we easily check that either $g_{\eta,k+1}$ or $g_{\eta,k+2}$ is bounded). The lower bound
is obtained by similar arguments. In particular, we have proved that if
$1/(\beta-1)$ is not an integer, then for any $\eta \in (0,1)$, we
have
\begin{eqnarray*}
i_\eta(w)&=&\inf\{k\ge 2\; : \;  g_{\eta,k} \hbox{ is bounded}\}\\
 &=&\inf\{k\ge 2\; : \;  (k-1)(\beta-1)>1\},
\end{eqnarray*}
which implies  Proposition \ref{calculi}.
\end{proof}
\begin{rem}\label{ipmdifferent}
Using similar arguments as the ones developed above, one can
construct examples of weight functions $w$ with $i_-(w)\neq i_+(w)$.
For instance, choosing $w(k)\sim k\exp(-\sqrt{2\log 2\log k})$, it
is not difficult to check that $i_-(w)=2$ whereas $i_+(w)=3$.
\end{rem}
We conclude this section by providing  an example of a weight sequence
whose indexes $i_\pm(w)$ are infinite.
\begin{cor}\label{ipminfini}
Let $\tilde{w}$ be a weight function  such that
$\tilde{w}(x):=x\exp(-\frac{\log x}{\log \log x})$ for $x$ large
enough. Then $i_\pm(\tilde{w}) = +\infty$.
 \end{cor}
\begin{proof} In view of Propositions \ref{calculi} and \ref{wequiv}, we just need to
  show that, for any $\alpha \in (0,1)$, the function $F:=(w\circ W^{-1})/(\tilde{w}\circ \tilde{W}^{-1})$ is eventually
  non-decreasing, where $w$ is defined by
  \eqref{defwalpha}.
Computing the derivative of $F$, we see that this property holds as
soon as
$$\tilde{w}'(x)\le w'\circ W^{-1}\circ \tilde{W}(x)\quad \mbox{ for $x$ large enough}.$$
Using  $W^{-1}(x)\le U_2(x)$  and  $w'$  non-increasing,
we get
$$w'\circ W^{-1}\circ \tilde{W}(x)\ge w'\circ U_2\circ \tilde{W}(x)\ge \frac{\beta(\log \tilde{W}(x))^{\beta-1}}{4\tilde{W}(x)}\qquad \mbox{ with $\beta:=1/\alpha$}.$$
Moreover, integrating by part, we get
$$ \tilde{W}(x)\sim \exp\left(\frac{\log x}{\log \log x}\right) \log \log x.$$
It follows that $$\tilde{w}'(x)\sim \exp\left(-\frac{\log x}{\log
\log x}\right) \sim \frac{\log \log x}{\tilde{W}(x)}\le w'\circ
W^{-1}\circ \tilde{W}(x),$$ which concludes the proof of the
corollary.
\end{proof}

\section{Coupling of three walks on the
half-line}\label{seccouplage}

In the rest of the paper, we assume that the weight function $w$
satisfies Assumption \ref{assumw} (i) and (ii) so we can use all the
results of the previous section. In order to study the VRRW $X$ on
$\Z$, we first look at the reflected VRRW $\bar{X}$ on the positive
half-line $\lin -1,\infty \irin$. The main idea is to compare this
walk with two simpler self-interacting processes $\tilde{X}$ and
$\widehat{X}$, which, in a way, "surround" the process we are
interested in. The study of $\tilde{X}$ and $\widehat{X}$ is
undertaken in Sections \ref{sectiontilde} and \ref{sectionhat}. The
estimates obtained concerning these two walks are then used in
Section \ref{sectionbar} to study the reflected VRRW $\bar{X}$.

\subsection{A general coupling result}

During the proof of Theorem \ref{locps}, we shall need to consider
processes whose transition probabilities depend, not only on the
adjacent site local time but also on its adjacent edge local time.
Furthermore, it will also be convenient to define processes starting
from arbitrary initial configurations of their edge/site local
times. To make this rigorous, we define the notion of \emph{state}.

\begin{defin}\label{defC}
We call \emph{state} any sequence $\kC=(z(x),n(x,x+1))_{x\in \Z}$ of
non-negative integers such that
$$n(x,x+1)\le z(x+1) \quad\hbox{ for
all $x\in \Z$.}$$
Given $\kC$ and some nearest neighbour path
$X=(X_n,n\ge 0)$ on $\Z$, we define its state
$\kC_n:=(Z_n(x),N_n(x,x+1))_{x\in \Z}$  at time $n$ by
\begin{equation}\label{defZandN}
Z_n(x):=z(x) +\sum_{i=0}^n \mathbf{1}_{\{X_i=x\}}
\quad\mbox{and}\quad N_n(x,x+1):=n(x,x+1)+\sum_{i=0}^{n-1}
\mathbf{1}_{\{X_{i}=x \mbox{ and } X_{i+1}=x+1\}},
\end{equation}
and we say that $\kC$ is the initial state of $X$. Thus $Z_n(x)$ is
the local time of $X$ at site $x$ and time $n$ whereas $N_n(x,x+1)$
corresponds to the local time on the oriented edge $(x,x+1)$ when we
start from $\kC$ (notice that
$\kC_0 \neq \kC$ since the site local time differs at $X_0$). We say
that $\kC$ is \emph{trivial} (resp. \emph{finite})
when all (resp. all but a finite number of)  the local times are $0$. Finally, we say that the state
$\mathcal{C}=(z(x),n(x,x+1))_{x\in \Z}$ is \emph{reachable} if
\begin{eqnarray*}
&(1)& \{x\in\Z\; : \; n(x,x+1)>0\} = \lin a,b-1\rin \hbox{ for some $a\le 0 \le b$,}\\
&(2)& z(x)= n(x,x+1)+n(x-1,x) \hbox{ for all $x\in\Z$.}
\end{eqnarray*}
\end{defin}
The terminology \textit{reachable} is justified by the following
elementary result, whose proof is left to the reader:
\begin{lem} A state $\mathcal{C}$ is reachable i.f.f. it can be
created from the trivial initial state by a finite path starting and
ending at zero (not counting the last visit at the origin for the
local time at site $0$).
\end{lem}

In order to compare walks with different transition mechanisms it is
convenient to construct them on the same probability space. To do
so, we always use the same generic construction which we now describe. Consider a sequence
$(U_i^x,x \in \Z,i\ge 1)$ of i.i.d. uniform random variables on
$[0,1]$ defined on some probability space
$(\Omega,\mathcal{F},\pp)$. Let $\kC$ be some fixed initial state.
Let $\mathbb{Q}$ be a probability measure on infinite nearest
neighbour paths on $\Z$ starting from $0$ (which may depend on $\kC$)
and write $\mathbb{Q}(x_0,\ldots,x_n)$ for the probability that a path
starts with $x_0,\ldots,x_n$. We construct on
$(\Omega,\mathcal{F},\pp)$ a random walk $X$ with image law
$\mathbb{Q}$ by induction in the following way:
\begin{itemize}
\item Set $X_0=0$.
\item $X_0,\ldots,X_{n}$ being constructed, if $Z_n(X_n)=i$, set
$$
X_{n+1}= \left\{
\begin{array}{ll}
X_n-1&\hbox{if $U_i^{X_n} \le
\mathbb{Q}(X_0,\ldots,X_n,X_n-1\;|\;X_0,\ldots,X_n)$,}\\
X_n+1&\hbox{otherwise,}
\end{array}
\right.
$$
where $Z_n$ stands for the local time of $X$ with initial state $\kC$ as in
Definition \ref{defC}.
\end{itemize}
This construction depends of the choice of $\kC =
(z(x),n(x,x+1))_{x\in\Z}$. In particular, if $z(x)>0$ for some
$x\in\Z$, then the random variables $U_1^{x},\ldots,U^x_{z(x)}$ are
not used in the construction.

 In the rest of the paper, all the walks considered are
constructed from the same sequence $(U_i^x)$ and with the same
initial state $\kC$. Hence, with a slight abuse of notation, we will
write $\pp_{\kC}$ to indicate that the walks are constructed using
the initial state $\kC$. Furthermore, if $\kC$ is the trivial state, we simply use the notation
$\pp_0$. Finally, since all the walks considered in the paper start
from $0$, we do not indicate the starting point in the notation for
the probability measure.

\medskip

Given a walk $X$, we denote its natural filtration by $\mathcal{F}_n :=
\sigma(X_0,\ldots,X_n)$. For $i,j,n\ge 0$ and $x\in \Z$, we define
the sets
\begin{eqnarray}\label{defAB}
\nonumber\mathcal{A}_{i,j}(n,x)&:=&\{X_n=x,\ Z_n(x-1)\ge i,\
Z_n(x+1)\le j \}\\
\mathcal{B}_{i,j}(n,x)&:=&\{X_n=x,\ Z_n(x-1)\le i,\ Z_n(x+1)\ge j
\}.
\end{eqnarray}
We also consider the stopping time
\begin{equation*}\label{defsigma}
\sigma(x,k):=\inf\{n\ge 0\; : \; Z_n(x)=k\}.
\end{equation*}
The following technical, yet very natural result, which is mainly equivalent to Lemma 4.1 of \cite{T1} enables us to compare walks with different transition probabilities.
\begin{lem}\label{propEERW} Let $\kC$ be some initial state and let $X,X'$ be two
nearest neighbours random walks (with possibly distinct mechanisms which may depend on
$\kC$) constructed on $(\Omega,\mathcal{F},\pp_{\kC})$. Assume that
the laws of $X$ and $X'$ are such that, for all $i,j,n,m\ge 0$ and
all $x\in \Z$, we have, $\pp_{\kC}$-a.s.
\begin{equation}\label{eqcouplage}
\pp_{\kC}\{X_{n+1}=x+1  \mid  \mathcal{F}_n,\
\mathcal{A}_{i,j}(n,x)\}\ \le\ \pp_{\kC}\{X'_{m+1}=x+1 \mid
\mathcal{F}'_m,\ \mathcal{B}'_{i,j}(m,x)\}
\end{equation}
(with the obvious $'$ notation for quantities related to $X'$).
Then, for all $x\in \Z$ and all $k\ge 0$  such that the stopping
times $\sigma(x,k)$ and $\sigma'(x,k)$ are both finite, we have
\begin{equation}\label{couple1}
Z_{\sigma(x,k)}(x-1)\geq Z'_{\sigma'(x,k)}(x-1) \quad\textrm{and}
\quad Z_{\sigma(x,k)}(x+1)\leq Z'_{\sigma'(x,k)}(x+1),
\end{equation}
and
\begin{equation}\label{implx_1}
X_{\sigma(x,k)+1}=x+1 \quad\Longrightarrow \quad
X'_{\sigma'(x,k)+1}=x+1.
\end{equation}
In the sequel, when \eqref{couple1} and \eqref{implx_1} hold, we
will say that $X$ is \emph{at the left} of $X'$ and write $X \prec
X'$.
\end{lem}

\begin{proof} In view of \eqref{eqcouplage}, if \eqref{couple1}
holds for some $(x,k)$, then so does \eqref{implx_1}. Hence, it
suffices to prove, by induction on $n\geq 0$, the assertion
\begin{equation}\label{hyprec0}
`` \forall x,k \mbox{ such that } \sigma(x,k)\le n,\;
\eqref{couple1}
 \mbox{ holds}. "
\end{equation}
This assertion is trivial for $n=0$ since both walks start with the
same initial state. Let us now assume that \eqref{hyprec0} holds for
some $n\geq0$. Let $(k_0,x_0)$ be such that $\sigma(x_0,k_0)=n+1$
and assume that $\sigma'(x_0,k_0)=m+1<\infty$. There are two cases.
Either this is the first visit to $x_0$ (i.e. $k_0=Z_0(x_0)+1$),
then $X_{n}=X'_{m}$ since both walks have the same starting point.
Otherwise, we are dealing with a subsequent visit to $x_0$. Applying
the recurrence hypothesis with $(k_0-1,x_0)$, it follows from
\eqref{implx_1} that
$$X_{\sigma(x_0,k_0-1)+1}=x_0+1 \quad\Longrightarrow\quad X'_{\sigma'(x_0,k_0-1)+1}=x_0+1.$$
Thus, in any case, we have
$$X_{n} \le X'_{m}\in \{x_0\pm 1\}.$$
If $X_{n} <X'_{m},$ then \eqref{couple1} clearly holds for
$(x_0,k_0)$ since
$Z'_{\sigma'(x_0,k_0)}(x_0-1)=Z'_{\sigma'(x_0,k_0-1)}(x_0-1)$ and
$Z_{\sigma(x_0,k_0)}(x_0+1)=Z_{\sigma(x_0,k_0-1)}(x_0+1)$. Assume
now that $X_{n}= X'_{m}=x_0-1$ (the case $x_0+1$ being similar).
Clearly, we have $Z'_{\sigma'(x_0,k_0)}(x_0+1)\ge
Z_{\sigma(x_0,k_0)}(x_0+1)$. It remains to  prove the converse
inequality for $x_0 - 1$. Denoting $i:=Z_n(x_0-1)$ and applying
\eqref{couple1} with $(x_0-1,i)$,  we find that, when
$\sigma'(x_0-1,i)<\infty$,
$$k_0-1= Z_{\sigma(x_0-1,i)}(x_0) \le Z'_{\sigma'(x_0-1,i)}(x_0).$$
Hence
$$ \sigma'(x_0,k_0-1)=m \le \sigma'(x_0-1,i).$$
This inequality trivially holds when $\sigma'(x_0-1,i)=\infty$) thus
$$Z'_{\sigma'(x_0,k_0)}(x_0-1)=Z_{m}(x_0-1)\le i = Z_n(x_0-1)=
Z_{\sigma(x_0,k_0)}(x_0-1).$$
This completes the proof of he lemma.
\end{proof}

\begin{cor}\label{corZinfty} Let $X,X'$ be two random walks such that $X \prec X'$.
\begin{itemize}
\item[(i)]
Let
$x_0:=\inf\{x \in \Z\; : \; Z'_\infty(x)=\infty\}$. Then,
$$Z_\infty(x)\le Z'_\infty(x)\qquad \textrm{for all }x\ge x_0.$$
In particular, if $X'$ localizes on a finite subset $\lin a,b \rin$,
then $\limsup X\le b$.
\item[(ii)] On the event $\{\lim_{n\rightarrow \infty} X_n=+\infty\}$, we have
$$Z'_\infty(x)\le Z_\infty(x)\qquad \textrm{for all }x\in \Z.$$
In particular, if $X'$ is recurrent, then  $X$ cannot diverge to
$+\infty$.
\end{itemize}
\end{cor}
\begin{proof} (i) We prove the result by induction on $x\ge x_0$.
There is nothing to prove for $x = x_0$ since
$Z'_\infty(x_0)=\infty$. Let us now assume that the result holds
for some $x-1\ge x_0$. Letting $k:=Z'_\infty(x)$, we just need to
prove that, on $\{k < \infty\}\cap\{ \sigma(x,k)<\infty\}$, the walk
$X$ never visits site $x$ after time $\sigma(x,k)$. First, since
$x_0$ is visited infinitely often by $X'$, in view of
\eqref{implx_1}, we find that $X_{\sigma(x,k)+1} =
X'_{\sigma'(x,k)+1}=x-1$. Moreover, if $n > \sigma(x,k)$ is such
that $X_n=x-1$ then $n =\sigma(x-1,j)$ for some $j\in \lin
Z_{\sigma(x,k)}(x-1), Z_\infty(x-1)\rin \subset \linf
Z'_{\sigma'(x,k)}(x-1), Z'_\infty(x-1)\rinf$ where we used
\eqref{couple1} and the recurrence hypothesis for the inclusion.
Recalling that $X'$ does not visit site $x$ after time
$\sigma'(x,k)$, we conclude, using  \eqref{implx_1} again, that
$X_{n+1}=X'_{\sigma'(x-1,j)+1}=x-2$. This entails that $X$ never
visits site $x$ after time $\sigma(x,k)$.

\medskip

(ii) By contradiction, assume that
$$n:=\inf\{i\ge 0\; : \; Z'_i(x)>Z_\infty(x) \mbox{ for some $x$}\}<\infty$$
and let $x_0 = X'_n$. Two cases may occur:
\begin{itemize}
\item $X'_{n-1}=x_0-1$. This means that $X'$ jumped from $x_0$ to $x_0-1$ at its previous visits to $x_0$ (\emph{i.e.} its $Z_\infty(x_0)$-th
visit). On the other hand, since $X$ is transient to the right, it
jumps from $x_0$ to $x_0+1$ at its $Z_\infty(x_0)$-th visit to
$x_0$. This contradicts \eqref{implx_1}.
\item $X'_{n-1}=x_0+1$. By definition of $n$ we have $k:=Z'_{n-1}(x_0+1)\le Z_\infty(x_0+1)$
hence $\sigma(x_0+1,k)<\infty$. Using \eqref{couple1} we get
$Z_{\sigma(x_0+1,k)}(x_0)\ge
Z'_{\sigma'(x_0+1,k)}(x_0)=Z_\infty(x_0)$ whereas \eqref{implx_1}
gives $X_{\sigma(x_0+1,k)+1}=X'_{n}=x_0$. This yields
$Z_{\sigma(x_0+1,k)+1}(x_0) > Z_\infty(x_0)$ which is absurd.
\end{itemize}
\end{proof}

\subsection{The three walks $\widetilde{X}$,$\bar{X}$ and $\widehat{X}$}
We define three nearest neighbour random walks on $\lin -1,\infty
\irin$, starting from some initial state $\kC$, which are denoted
respectively by $\widetilde{X}, \bar{X}$ and $\widehat{X}$. All the
quantities referring to $\widetilde{X}$ (resp. $\bar{X}$,
$\widehat{X}$) are denoted with a tilde (resp. bar, hat). The three
walks are reflected at $-1$ \emph{i.e.},
$$\pp_\mathcal{C}\!\{\bar{X}_{n+1}\!=\!0 \mid  \bar{\mathcal{F}}_n, \bar{X}_n\!=\!-\!1\}
=\pp_\mathcal{C}\!\{\widetilde{X}_{n+1}\!=\!0 \mid
\widetilde{\mathcal{F}}_n,\widetilde{X}_n\!=\!-\!1\}
=\pp_\mathcal{C}\!\{\widehat{X}_{n+1}\!=\!0 \mid
\widehat{\mathcal{F}}_n,\widehat{X}_n\!=\!-\!1\}=1$$ and the
transition probabilities are given by the following rules:
\begin{itemize}
\item The walk $\bar{X}$ is a vertex reinforced random walk with weight $w$ reflected at $-1$, \emph{i.e.} for all $x\ge 0$,
\begin{equation}\label{transi_barX}
\pp_\mathcal{C}\{\bar{X}_{n+1}=x-1 \mid \mathcal{\bar{F}}_n,
\bar{X}_n =
x\}=\frac{w(\bar{Z}_n(x-1))}{w(\bar{Z}_n(x-1))+w(\bar{Z}_n(x+1))}.
\end{equation}

\item The walk $\widetilde{X}$  is a "mix" between an oriented edge-reinforced and a vertex-reinforced random walk:
when at site $x$, the walk makes a jump to the left with a
probability proportional to a function of the local time at the site
$x-1$ whereas it jumps to the right with a probability proportional
to a function of the local time on the oriented edge $(x,x+1)$. More
precisely, for $x\geq 0$,
\begin{equation}\label{transi_tildeX}
\pp_\mathcal{C}\{\widetilde{X}_{n+1}=x-1 \mid
\widetilde{\mathcal{F}}_n, \widetilde{X}_n =
x\}=\frac{w(\widetilde{Z}_n(x-1))}{w(\widetilde{Z}_n(x-1))+w(\widetilde{N}_n(x,x+1))}.
\end{equation}

\item The transition mechanism of the third walk $\widehat{X}$ is a
bit more complicated.  Similarly to the previous walk, $\widehat{X}$
jumps to the left with a probability proportional to a function of
the local time at the site on its left whereas it jumps to the right
with a probability proportional to a (different) function of the
local time on the oriented edge on its right. However, we do not
directly use the weight function $w$ because we want to increase the
reinforcement induced by the local time of the right edge. In order
to do so, we fix $\varepsilon>0$ small enough such that
$i_+(w)=i_{1/2+3\varepsilon}(w)$. Next, we consider a function $f :=
f_{1/2+2\varepsilon}$ as in Lemma \ref{WL} (\emph{i.e.} a function
satisfying (a)-(d) of Lemma \ref{WL} with $\eta =
1/2+2\varepsilon$). Given these two parameters, the transition
probabilities of $\widehat{X}$ are defined by
\begin{equation}\label{transi_hatX}
\pp_\mathcal{C}\{\widehat{X}_{n+1}=x-1 \mid \widehat{\mathcal{F}}_n,
\widehat{X}_n = x\} = \left\{
\begin{array}{ll}
\frac{w(\widehat{Z}_n(-1))}{\vphantom{X^{X^X}}w(\widehat{Z}_n(-1))+w\left(\widehat{N}_n(0,1)+f(\widehat{N}_n(0,1))\right)}&\hbox{ if $x = 0$,}\\
\frac{\vphantom{X^{{X^X}^X}}w(\widehat{Z}_n(x-1))}{\vphantom{X^{X^X}}w(\widehat{Z}_n(x-1))+w\left((1+\varepsilon)\widehat{N}_n(x,x+1)\right)}&\hbox{ if $x > 0$.}\\
\end{array}
\right.
\end{equation}
Comparing these transition probabilities with those of
$\widetilde{X}$, the edge local time $N(0,1)$ is slightly increased
by $f(N(0,1)) = o(N(0,1))$ whereas the edge local times $N(x,x+1)$
are multiplied by $1+\varepsilon$ for $x\geq 1$.
\end{itemize}

\begin{rem} \
\begin{enumerate}
\item[\textup{(a)}] Let us emphasize the fact that the laws of the three walks
depend on the initial state $\kC$ since the local times $Z_n(x)$ and
$N_n(x,x+1)$ depend upon it.
\item[\textup{(b)}] We should rigourously write $\widehat{X}^{\varepsilon,f}$
instead of $\widehat{X}$ since the law of the walk depends on the
choice of $(\varepsilon,f)$. However, these two parameters depend,
in turn, only on the weight function $w$ which is fixed throughout
the paper. For the sake of clarity, we keep the notation without any
superscript.
\end{enumerate}
\end{rem}

\subsection{Coupling between $\widetilde{X}$, $\bar{X}$ and $\widehat{X}$}

For any random walk, the local time at site $x$ is equal (up to an
initial constant) to the sum of the local times of the ingoing edges
adjacent to $x$ since the walk always reaches $x$ through one of
these edges. Hence, looking at the definition of $\widetilde{X}$ and
$\bar{X}$, we see that the reinforcements schemes give a stronger
"push to the right" for $\bar{X}$ than for $\widetilde{X}$ so it is
reasonable to expect $\widetilde{X}$ to be at the left of $\bar{X}$.
This is indeed the case:

\begin{lem}\label{proptildeleftbar} For any initial state $\mathcal{C}$,  under $\pp_\mathcal{C}$, we have $\widetilde{X} \prec \bar{X}$.
\end{lem}
\begin{proof} We just need to show that \eqref{eqcouplage} holds with $\widetilde{X}$ and
$\bar{X}$.  Define $\widetilde{\mathcal{A}}_{i,j}(n,x)$ and
$\bar{\mathcal{B}}_{i,j}(n,x)$ as in \eqref{defAB}. On the one hand,
for $x\ge 0$, we have
\begin{equation*}
\pp_\mathcal{C}\{\bar{X}_{n+1}=x-1 \ | \
\bar{\mathcal{F}}_n,\bar{\mathcal{B}}_{i,j}(n,x) \} =
\frac{w(\bar{Z}_n(x-1))}{w(\bar{Z}_n(x-1))+w(\bar{Z}_n(x+1))}
\mathbf{1}_{\{\bar{\mathcal{B}}_{i,j}(n,x)\}}  \le
\frac{w(i)}{w(i)+w(j)}.
\end{equation*}
On the other hand, since we have by definition of a state that
$\widetilde{N}_0(x,x+1)\le \widetilde{Z}_0(x+1)$ for all $x$, we
also have $\widetilde{N}_n(x,x+1)\le \widetilde{Z}_n(x+1)$ for any
$x,n$ and thus
\begin{equation*}
\pp_\mathcal{C}\{\widetilde{X}_{n+1}=x-1 \ | \
\widetilde{\mathcal{F}}_n,\widetilde{\mathcal{A}}_{i,j}(n,x) \} =
\frac{w(\widetilde{Z}_n(x-1))}{w(\widetilde{Z}_n(x-1))+w(\widetilde{N}_n(x,x+1))}
\mathbf{1}_{\{\widetilde{\mathcal{A}}_{i,j}(n,x)\}} \geq
\frac{w(i)}{w(i)+w(j)},
\end{equation*}
which proves $\eqref{eqcouplage}$.
\end{proof}

\medskip  Unfortunately, as we cannot \emph{a priori} compare the quantity $(1+\varepsilon)N_n(x+1,x)$
with $Z_n(x)$ nor $N_n(0,1)+f(N_n(0,1))$ with $Z_n(1)$, there is no
direct coupling between $\bar{X}$ and $\widehat{X}$. However, we can
still define a "good event" depending only on $\widehat{X}$ on which
$\bar{X}$ is indeed at the left of $\widehat{X}$ with positive
probability. For $L,M\ge 0$, set
\begin{equation}\label{defEKN}
\widehat{\mathcal{E}}(L,M)= \left\{\exists K\le L,\; \forall n\ge M,\;  \begin{array}{l}  \widehat{Z}_n(1)\le \widehat{N}_n(0,1)+f(\widehat{N}_n(0,1)) \\
 \forall x \in \lin 2, K \rin, \; \widehat{Z}_n(x)\le (1+\varepsilon)\widehat{N}_n(x-1,x) \\
 \forall x \ge K, \; \widehat{Z}_n(x)=\widehat{Z}_{M}(x)\\
\end{array}\right\}.
\end{equation}

\begin{lem}\label{probapositive} Let $\mathcal{C}$ be any initial state.
\begin{enumerate}
\item[\textup{(i)}] Under $\pp_{\kC}$, we have $\bar{X}\prec \widehat{X}$ on
$\widehat{\mathcal{E}}(L,0)$ (meaning that \eqref{couple1} and
\eqref{implx_1} hold on this event) and
\begin{equation}\label{ek0}
\widehat{\mathcal{E}}(L,0) \subset\{\bar{X} \mbox{ never visits site
$L$}\}.
\end{equation}
\item[\textup{(ii)}] Assume that $\pp_\mathcal{C}\{\widehat{\mathcal{E}}(L,M)\}>0$ for some $L,M\geq 0$. Then, under
$\pp_\mathcal{C}$, with positive probability, the walk $\bar{X}$
ultimately stays confined in the interval $\lin -1, L-1 \rin$.
\end{enumerate}
\end{lem}

\begin{proof} Concerning the first part of the lemma, the fact that $\bar{X}\prec \widehat{X}$ on
$\widehat{\mathcal{E}}(L,0)$ follows from the definition of
$\widehat{\mathcal{E}}(L,0)$ combined with
\eqref{transi_barX},\eqref{transi_hatX}
using the same argument as in the previous lemma. Moreover, we have
$\widehat{\mathcal{E}}(L,0) \subset\{\widehat{X} \mbox{ never visits
site $L$}\}$. Hence \eqref{ek0} is a consequence of Corollary
\ref{corZinfty}.

We now prove (ii). We introduce an auxiliary walk $X^*$ on $\lin
-1,\infty \irin $ such that $\bar{X}\prec X^*$ and  coinciding with
$\widehat{X}$ on a set of positive probability. The walk $X^*$ is
reflected at $-1$ and with transition probabilities given for $x\geq 0$ by
\begin{equation*}
\pp_\mathcal{C}\{X^{*}_{n+1}=x-1 \; | \;
\mathcal{F}^{*}_n,X^{*}_n=x\}=\frac{w(Z^{*}_n(x-1))}{w(Z^{*}_n(x-1))+w(V^{*}_n(x+1))},
\end{equation*}
where the functional $V^*$ is defined by
\begin{equation*}
V^*_n(x) := \left\{
\begin{array}{ll}
\max(Z^*_n(1),N^*_n(0,1)+f(N^*_n(0,1))) & \mbox{ for $x= 1$}\\
\max(Z^*_n(x),(1+\varepsilon)N^*_n(x-1,x)) & \mbox{ for $x\neq 1$}.
\end{array}
\right.
\end{equation*}
Since $V^*_n\ge Z^*_n$, it follows clearly that $\bar{X}\prec
X^{*}$. Now set
\begin{equation*}
\mathcal{G}:=\widehat{\mathcal{E}}(L,M)\cap\{\forall n\ge 0,
X^{*}_n=\widehat{X}_n\}.
\end{equation*}
On $\widehat{\mathcal{E}}(L,M)$, there exists some $K\le L$ such
that, for all $n> M$,
\begin{equation*}
\widehat{X}_n \in \lin -1, K-1\rin
\quad\hbox{and}\quad\widehat{V}_n(x)= \left\{
\begin{array}{ll}
\widehat{N}_n(0,1)+f(\widehat{N}_n(0,1))) & \mbox{ for $x= 1$,}\\
(1+\varepsilon)\widehat{N}_n(x-1,x)) & \mbox{ for $x\in \lin 2, K
\rin$.}
\end{array}\right.
\end{equation*}
Therefore
\begin{equation*}
\mathcal{G}= \widehat{\mathcal{E}}(L,M)\cap \{\forall n\le M,\,
X^{*}_n=\widehat{X}_n\}.
\end{equation*}
By ellipticity, we have a.s. $\pp_\mathcal{C}\{\forall n\le M,\,
X^{*}_n=\widehat{X}_n\ | \ \mathcal{\widehat{F}}_{M}\}>0$.
Conditionally on $\mathcal{\widehat{F}}_{M}$, the events $\{\forall
n\le M,\; X^{*}_n=\widehat{X}_n\}$ and $\widehat{\mathcal{E}}(L,M)$
are independent. Assuming that
$\pp_\mathcal{C}\{\widehat{\mathcal{E}}(L,M)\}>0$, we deduce that
$\pp_\mathcal{C}\{\mathcal{G}\}>0$. Moreover, on $\mathcal{G}$, we
have $Z^{*}_\infty(x)= \widehat{Z}_\infty(x)=\widehat{Z}_{M}(x)$ for
all $x\ge L$ (i.e. $X^{*}$ stays in the interval $\lin -1, L-1 \rin$
after time $M$). Using  $\bar{X}\prec X^*$, Corollary
\ref{corZinfty} gives
\begin{equation*}\label{ZinftyXhatX}
\mathcal{G}\subset\{\forall x\ge L, \;
\bar{Z}_\infty(x)\le\widehat{Z}_{M}(x)\},
\end{equation*}
which implies
$$\pp_\kC\{\bar{X} \mbox{ eventually remains in the interval
$\lin -1, L-1 \rin$}\}\ge \pp_\kC\{\mathcal{G}\}>0.$$

\end{proof}

\section{The walk $\widetilde{X}$}\label{sectiontilde}
We now study the asymptotic behaviour of $\widetilde{X}$. This walk
is the easiest to analyse among those defined in the previous
section and it is possible to obtain a precise description of the
localization set. In fact, we can even show recurrence when the walk does
not localize.

We introduce some notation to help make the proof more readable by
removing unimportant constants. Given two (random) processes
$A_n,B_n$, we will write $A_n\equiv B_n$ when $A_n-B_n$ converges
a.s. to some (random) finite constant. Similarly we write
$A_n\lesssim B_n$ when $\limsup A_n-B_n$ is finite a.s..

\begin{prop}\label{ERRW} Let  $\mathcal{C}$ be a finite state. Recall that $\widetilde{R}$ denotes the set of sites visited
i.o. by $\widetilde{X}$. We have
\begin{equation*}
\lin -1,j_-(w)-1\rin \subset  \widetilde{R} \subset \lin
-1,j_+(w)-1\rin \qquad \pp_\kC\mbox{-a.s.}
\end{equation*}
In particular, the walk is either recurrent or localizes a.s.
depending on the finiteness of $j_\pm(w)$.
\end{prop}

\begin{proof}
First, it is easy to check that the walk $\widetilde{X}$ is at the
left (in the sense of Proposition \ref{propEERW}) of an oriented
edge reinforced random walk with weight $w$ reflected at $-1$ that
is, a random walk which jumps from $x$ to $x+ 1$ with probability
proportional to $w(N_n(x,x+1))$ (where $N_n(x,x+1)$ is defined by
\eqref{defZandN}) and from $x$ to $x-1$ with probability
proportional to $w(N_n(x,x-1))$ where $N_n(x,x-1)$ is simply the
number of jumps from $x$ to $x-1$ before time $n$ (but without any
additional initial constant). Such a walk
 can be constructed from a family $(\mathcal{U}_x,x\geq 0)$ of
independent generalized Pólya $w$-urns where the sequence of draws in
the urn $\mathcal{U}_x$ corresponds to the sequence of jumps to
$x-1$ or $x+1$ when the walk is at site $x$. Using this
representation, Davis \cite{Dav} showed that, if $\kC$ is finite,
the oriented edge reinforced random walk is recurrent as soon as
$\sum 1/w(k)=\infty$ (more precisely, in \cite{Dav}, recurrence is
established for the non-oriented version of the edge reinforced walk
but the same proof also applies to the oriented version and is even
easier in that case).

In view of Corollary \ref{corZinfty}, it follows from the recurrence
of the oriented edge reinforced random walk that $\widetilde{X}$
cannot tend to infinity hence there exists at least one site which
is visited infinitely often. Next, noticing that
\begin{equation*}
\sum_{n=0}^\infty \pp_\kC\{\widetilde{X}_{n+1}=x-1 \; | \;
\widetilde{\mathcal{F}}_n\} \ \ge \ \sum_{n=0}^\infty
\frac{w(0)\mathbf{1}_{\{\widetilde{X}_n=x\}}}{w(0)+w(\widetilde{N}_n(x,x+1))}
\ \ge
 \ \sum_{n=n_0}^{\widetilde{Z}_{\infty}(x)} \frac{w(0)}{w(0)+w(n)}
\end{equation*}
the conditional Borel-Cantelli Lemma implies that if $x$ is visited
i.o., then so will $x-1$. By induction we deduce that $-1$ is
visited i.o. a.s. Now, we have to prove that any site $x \leq j_-(w)$ is
 visited i.o. but that $j_+(w)+1$ is not. More precisely, we show by
induction that for each $j\geq 1$:
\begin{equation}\label{hyprec}
\forall \alpha\in (0,1/2),\quad
\Phi_{1/2-\alpha,j}(\widetilde{Z}_k(j-1))
 \;\lesssim\; \widetilde{N}_k(j-1,j) \;\lesssim\; \Phi_{1/2+\alpha,j}(\widetilde{Z}_k(j-1)) \quad\mbox{
 a.s.}\end{equation}
where $(\Phi_{\eta,j})_{\eta\in (0,1), j\ge 1}$ is the sequence of
functions defined in \eqref{defphi}. For $x \ge 0$ ,  define
\begin{equation*}
\widetilde{M}_n(x) \; := \;\sum_{k=0}^{n-1}
\frac{\mathbf{1}_{\{\widetilde{X}_k=x\textrm{ and }\widetilde{X}_{k+1}=x+
1\}}}{w(\widetilde{N}_k(x,x+1))}-\sum_{k=0}^{n-1}
\frac{\mathbf{1}_{\{\widetilde{X}_k=x\textrm{ and }\widetilde{X}_{k+1}=x-
1\}}}{w(\widetilde{Z}_k(x- 1))}.
\end{equation*}
It is well known and easy to check that $(\widetilde{M}_n(x),n\ge
0)$ is a martingale bounded in $L^2$ which converges a.s. to a
finite random variable \emph{c.f.} for instance \cite{T2,BSS}.
Recalling the definition of $W$ given in \eqref{defW} we also have
$$W(n)\equiv \sum_{i=1}^{n-1} \frac{1}{w(i)}.$$
Hence, we get
$$\widetilde{M}_{n}(0)\equiv W(\widetilde{N}_n(0,1))-W(\widetilde{Z}_n(-1))$$
and the convergence of the martingale $\widetilde{M}_{n}(0)$
combined with Lemma \ref{limborne} yields
$$\lim_{n \rightarrow \infty}
\frac{\widetilde{N}_n(0,1)}{\widetilde{Z}_n(-1)}= 1
\quad\mbox{$\pp_{\kC}$-a.s.}$$ Noticing that $\widetilde{Z}_n(0)\sim
\widetilde{N}_n(0,1)+\widetilde{Z}_n(-1)$ and recalling that
$\Phi_{\eta,1}(x)=\eta x$ we conclude that \eqref{hyprec} holds for
$j=1$.

\medskip

Fix $j\ge 1$ and assume that \eqref{hyprec} holds for $j$. If
$\widetilde{N}_\infty(j-1,j)$ is finite, then
$\widetilde{Z}_\infty(j)$ and $\widetilde{N}_\infty(j,j+1)$ are
necessarily also finite so \eqref{hyprec} holds for $j+1$. Now assume that $\widetilde{N}_\infty(j-1,j)$ is infinite which, in
view of \eqref{hyprec}, implies that $\widetilde{Z}_\infty(j-1)$ is
also infinite and that
$$\lim_{t\to\infty} \Phi_{1/2+\alpha,j}(t)=\infty \quad \mbox{ for any $\alpha\in
(0,1/2)$.} $$
 Besides, the convergence of the
martingale $\widetilde{M}_n(j)$ yields
\begin{equation}\label{eqti}
W(\widetilde{N}_n(j,j+1))\equiv\sum_{k=0}^{n-1}
\frac{\mathbf{1}_{\{\widetilde{X}_k=j\textrm{ and }\widetilde{X}_{k+1}=j-
1\}}}{w(\widetilde{Z}_k(j- 1))}.
\end{equation}
 According to Lemma
\ref{lemmtech}, we have
$$\lim_{t\to\infty} \left( \Phi_{1/2+\alpha',j}(t)-\Phi_{1/2+\alpha,j}(t)\right)=\infty\quad \mbox{ for any $0<\alpha<\alpha'<1/2$,}$$
hence we get from \eqref{hyprec} that for $k$ large enough
$\widetilde{Z}_k(j-1)\ge
\Phi_{1/2+\alpha,j}^{-1}(\widetilde{N}_k(j-1,j))$. Combining this
with \eqref{eqti} yields
\begin{equation*}
W(\widetilde{N}_n(j,j+1))\lesssim
\sum_{k=0}^{\widetilde{N}_n(j-1,j)}
\frac{1}{w(\Phi_{1/2+\alpha,j}^{-1}(k))}.
\end{equation*}
Recalling the definition of the sequence $(\Phi_{\eta,j})_{j\ge 1}$
we obtain
\begin{equation}\label{eqti2}
W(\widetilde{N}_n(j,j+1))\lesssim
W(\Phi_{1/2+\alpha,j+1}(\widetilde{N}_n(j-1,j))).
\end{equation}
Thus, for $\alpha'>\alpha$ and for $k$ large enough, using Lemmas
\ref{limborne} and \ref{lemmtech}, we get
$$\widetilde{N}_k(j,j+1)\le 2\Phi_{1/2+\alpha,j+1}(\widetilde{N}_k(j-1,j))\le \Phi_{1/2+\alpha',j+1}(\widetilde{N}_k(j-1,j))\le \Phi_{1/2+\alpha',j+1}(\widetilde{Z}_k(j))$$
provided that $\lim_{t\to\infty}\Phi_{1/2+\alpha,j+1}(t) = \infty$.
When the previous limit is finite, it follows readily from
\eqref{eqti2} that $\widetilde{N}_\infty(j,j+1) < \infty$. Thus, in
any case, we obtain the required upper bound
\begin{equation}\label{eqti3}
\widetilde{N}_k(j,j+1) \lesssim
\Phi_{1/2+\alpha,j+1}(\widetilde{Z}_k(j)).
\end{equation}
Concerning the lower bound, there is nothing to prove if
$\lim_{t\to\infty}\Phi_{1/2-\alpha,j+1}(t) < +\infty$. Otherwise, it
follows from \eqref{eqti3} and Lemma \ref{lemmtech} that
$\widetilde{N}_k(j,j+1)=o(\widetilde{Z}_k(j))$. Moreover, using
exactly the same argument as before, we find that for $k$ large
enough
\begin{equation*}
\widetilde{N}_k(j,j+1)\ge
\Phi_{1/2-\alpha,j+1}(\widetilde{N}_k(j-1,j)).
\end{equation*}
Noticing that $\widetilde{N}_k(j-1,j) \sim
(\widetilde{Z}_k(j)-\widetilde{N}_k(j,j+1)) \sim
\widetilde{Z}_k(j)$, we conclude using again Lemma \ref{lemmtech}
that for $\alpha' > \alpha$ and for $k$ large enough,
\begin{equation*}
\widetilde{N}_k(j,j+1)\geq
\Phi_{1/2-\alpha',j+1}(\widetilde{Z}_k(j)),
\end{equation*}
which yields the lower bound of \eqref{hyprec}.

 Finally, choosing $\alpha>0$ small enough such that
$\lim_{t\to \infty} \Phi_{1/2+\alpha,j_+(w)}(t)<\infty$ we deduce
that $\widetilde{N}_\infty(j_+(w)-1,j_+(w))$ is finite hence
$\widetilde{Z}_\infty(j_+(w))$ is also finite. Conversely,
\eqref{hyprec} entails by a straightforward induction that
$\widetilde{Z}_\infty(j)=\infty$ for $j< j_-(w)$.
\end{proof}

\section{The walk $\widehat{X}$}\label{sectionhat}
We now turn our attention towards the walk $\widehat{X}$ which is
more delicate to analyse than the previous process so we only obtain
partial results concerning its asymptotic behaviour. In view of
Lemma \ref{probapositive}, we are mainly interested in finding the
smallest integer $L$ such that
$\pp_{\mathcal{C}}\{\widehat{\mathcal{E}}(L,M)\}>0$ for some $M$.
The purpose of this section is to prove the proposition below which
provides an upper bound for $L$ which is optimal when
$j_-(w)=j_+(w)$.

\begin{prop}\label{modERRW} Assume that $j_+(w)<\infty$. Then,
 for any initial state $\mathcal{C}$, there exists $M\ge 0$ such that
\begin{equation}\label{c1}
\pp_{\mathcal{C}}\{\widehat{\mathcal{E}}(j_+(w),M) \}>0.
\end{equation}
Moreover, there exists a reachable initial state
$\mathcal{C}'=(z'(x),n'(x,x+1))_{x\in \Z}$ which is zero outside of
the interval
  $\lin -1, j_+(w) \rin$   and with $n'(0,1)\ge n'(-1,0)$  such that
\begin{equation}\label{c2}
\pp_{\mathcal{C}'}\{\widehat{\mathcal{E}}(j_+(w),0)\}>3/4.
\end{equation}
\end{prop}

One annoying difficulty studying $\widehat{X}$ is that we cannot
easily exclude  the walk diverging to $+\infty$ on a set of
non-zero probability. In order to bypass this problem, we first
study the walk on a bounded interval. More precisely, for $L> 1$, we
define the walk $\widehat{X}^{\scriptscriptstyle{\! L}}$ on $\lin
-1,L \rin$ which is reflected at the boundary sites $-1$ and $L$,
with the same transition probabilities as $\widehat{X}$ in the
interior of the interval:
\begin{equation*}
\pp_\mathcal{C}\{\widehat{X}^{\scriptscriptstyle{\! L}}_{n+1}=x-1
\mid \widehat{\mathcal{F}}^{\scriptscriptstyle{\! L}}_n,\;
\widehat{X}^{\scriptscriptstyle{\! L}}_{n}=x\} =
\left\{
\begin{array}{ll}
0 &\hbox{ if $x = -1$,}\\
\frac{w(\widehat{Z}^{\scriptscriptstyle{\! L}}_n(-1))}{\vphantom{X^{X^X}}w(\widehat{Z}^{\scriptscriptstyle{\! L}}_n(-1))+w\left(\widehat{N}^{\scriptscriptstyle{\! L}}_n(0,1)+f(\widehat{N}^{\scriptscriptstyle{\! L}}_n(0,1))\right)}&\hbox{ if $x = 0$,}\\
\frac{\vphantom{X^{{X^X}^X}}w(\widehat{Z}^{\scriptscriptstyle{\! L}}_n(x-1))}{\vphantom{X^{X^X}}w(\widehat{Z}^{\scriptscriptstyle{\! L}}_n(x-1))+w\left((1+\varepsilon)\widehat{N}^{\scriptscriptstyle{\! L}}_n(x,x+1)\right)}&\hbox{ if $x \in \lin 1,L-1\rin$,}\\
1 &\hbox{ if $x = L$.}
\end{array}
\right.
\end{equation*}
The proof of Proposition \ref{modERRW} relies on the following lemma
which estimates the edge/site local times of
$\widehat{X}^{\scriptscriptstyle{\! L}}$.
\begin{lem}\label{lemhatXK} Let $\kC$ be an initial state and  $L>1$. For $n$ large enough, we have
\begin{equation}\label{desgauche}
 \widehat{N}^{\scriptscriptstyle{\! L}}_n(-1,0) \;\le\; \widehat{N}^{\scriptscriptstyle{\! L}}_n(0,1)\qquad \pp_\mathcal{C} \mbox{-a.s.}
\end{equation}
Moreover, for $\eta \in(1/2+\varepsilon,1)$ and $j \in \lin
0,L-1\rin$,
\begin{equation}\label{hyprec2}
\widehat{N}^{\scriptscriptstyle{\! L}}_n(j,j+1) \;\lesssim\;
\Phi_{\eta,j+1}\big(\widehat{Z}^{\scriptscriptstyle{\!
L}}_n(j)\big).
\end{equation}
\end{lem}

\begin{proof} The proof is fairly similar to that
of Proposition \ref{ERRW}. First, since
$\widehat{X}^{\scriptscriptstyle{\! L}}$ has compact support, the
set $\widehat{R}^{\scriptscriptstyle{\! L}}$ of sites visited
infinitely often by the walk is necessarily not empty. Furthermore,
noticing that $\sum 1/w((1+\varepsilon)n)$ is infinite since $w$ is
regularly varying, the same arguments as those used for dealing with
$\widetilde{X}$ show that $\widehat{X}^{\scriptscriptstyle{\! L}}$
visits site $0$ infinitely often a.s.

We first prove \eqref{desgauche} together with \eqref{hyprec2} for
$j=0$. As before, it is easily checked that
\begin{eqnarray*}
\widehat{M}^{\scriptscriptstyle{\! L}}_n(0) &:=& \sum_{k=0}^{n-1}
\frac{\mathbf{1}_{\{\widehat{X}^{\scriptscriptstyle{\!
L}}_k=0\textrm{ and }\widehat{X}^{\scriptscriptstyle{\! L}}_{k+1}=
1\}}}{w(\widehat{N}^{\scriptscriptstyle{\!
L}}_k(0,1)+f(\widehat{N}^{\scriptscriptstyle{\!
L}}_k(0,1)))}-\sum_{k=0}^{n-1}
\frac{\mathbf{1}_{\{\widehat{X}^{\scriptscriptstyle{\!
L}}_k=0\textrm{ and
}\widehat{X}^{\scriptscriptstyle{\! L}}_{k+1}=- 1\}}}{w(\widehat{Z}^{\scriptscriptstyle{\! L}}_k(-1))}\\
\end{eqnarray*}
is a martingale bounded in $L^2$ with converges to some finite
constant. Besides, recalling the definitions of $W$ and $W_f$, we
have
\begin{equation}\label{martmod0}
\widehat{M}^{\scriptscriptstyle{\! L}}_n(0)\equiv
W_f(\widehat{N}^{\scriptscriptstyle{\!
L}}_n(0,1))-W(\widehat{Z}^{\scriptscriptstyle{\! L}}_n(-1)).
\end{equation}
Since $0$ is visited infinitely often and since $W$ and $W_f$ are
unbounded, Equation \eqref{martmod0} implies that $-1$ and $1$ are
also visited infinitely often a.s. Recalling that $f$ satisfies (c)
of Lemma \ref{WL}, Lemma \ref{limWWchapeau} entails
\begin{equation}\label{eqlim-11}
\lim_{n \rightarrow \infty} \frac{\widehat{N}^{\scriptscriptstyle{\!
L}}_n(0,1)}{\widehat{Z}^{\scriptscriptstyle{\! L}}_n(-1)}=1 \qquad
\pp_\mathcal{C}\mbox{-a.s.}
\end{equation}
Using $\widehat{Z}^{\scriptscriptstyle{\! L}}_n(0)\sim
\widehat{Z}^{\scriptscriptstyle{\!
L}}_n(-1)+\widehat{N}^{\scriptscriptstyle{\! L}}_n(0,1)$, we find
for $\delta>1/2$ and for $n$ large enough,
\begin{equation}\label{eqj1}
\widehat{N}^{\scriptscriptstyle{\! L}}_n(0,1) \;\le\;
\delta\widehat{Z}^{\scriptscriptstyle{\! L}}_n(0) \; = \;
\Phi_{\delta,1}(\widehat{Z}^{\scriptscriptstyle{\! L}}_n(0)),
\end{equation}
 which, in particular, proves \eqref{hyprec2} for $j=0$. Moreover,
using  $\widehat{N}^{\scriptscriptstyle{\! L}}_n(-1,0)\le
\widehat{Z}^{\scriptscriptstyle{\! L}}_n(-1)+c$ for some constant
$c$ depending only on $\kC$, the fact that $W(x+c)-W(x)$ tends to
$0$ at infinity and recalling that $f$ satisfies (d) of Lemma
\ref{WL}, we deduce from \eqref{martmod0} that
\begin{equation*}
\lim_{n\rightarrow \infty} W(\widehat{N}^{\scriptscriptstyle{\!
L}}_n(0,1))-W(\widehat{N}^{\scriptscriptstyle{\! L}}_n(-1,0))=\infty
\qquad \pp_\mathcal{C}\mbox{-a.s.}
\end{equation*}
Since $W$ is non-decreasing, this shows that \eqref{desgauche}
holds.

We now prove \eqref{hyprec2} by induction on $j$. The
same martingale argument as before shows that
\begin{equation}\label{martmodi}
W_{\varepsilon}(\widehat{N}^{\scriptscriptstyle{\!
L}}_n(x,x+1))\equiv\sum_{k=0}^{n-1}
\frac{\mathbf{1}_{\{\widehat{X}^{\scriptscriptstyle{\!
L}}_k=x\textrm{ and }\widehat{X}^{\scriptscriptstyle{\! L}}_{k+1}=x-
1\}}}{w(\widehat{Z}^{\scriptscriptstyle{\! L}}_k(x- 1))}\qquad
\mbox{ for $x\in \lin 1,L-1\rin$},
\end{equation}
where we recall the notation $W_{\varepsilon} := W_\psi$ for
$\psi(x):=\varepsilon x$. Assume that \eqref{hyprec2} holds for
$j-1\in \lin 0, L-2 \rin$ and fix  $\eta \in (1/2+\varepsilon,1)$.
If $\widehat{N}^{\scriptscriptstyle{\! L}}_\infty(j-1,j)$ is finite,
then $\widehat{N}^{\scriptscriptstyle{\! L}}_\infty(j,j+1)$ is also
finite and the proposition holds for $j$. Hence, we assume that
$\widehat{N}^{\scriptscriptstyle{\! L}}_\infty(j-1,j) $ and
$\widehat{N}^{\scriptscriptstyle{\! L}}_\infty(j,j+1)$ are both
infinite. If $j=1$, we get, using \eqref{eqj1}, that for $n$ large enough,
$$\widehat{Z}^{\scriptscriptstyle{\! L}}_n(0)\;\geq\;  \frac{(1+2\varepsilon)}{\eta}\widehat{N}^{\scriptscriptstyle{\! L}}_n(0,1) = (1+2\varepsilon)\Phi_{\eta,1}^{-1}\big(\widehat{N}^{\scriptscriptstyle{\! L}}_n(0,1)\big).$$
On the other hand, if $j>1$, recalling that $\Phi_{\beta,j}(\lambda
t) \;\lesssim\; \Phi_{\alpha,j}(t)$ for $\alpha>\beta$ and
$\lambda>0$, we get using the recurrence hypothesis with $\eta'\in
(1/2+\varepsilon, \eta)$
$$\widehat{Z}^{\scriptscriptstyle{\! L}}_k(j- 1)\;\gtrsim\; \Phi_{\eta',j}^{-1}\big(\widehat{N}^{\scriptscriptstyle{\! L}}_k(j-1,j)\big)\;\gtrsim\;  (1+2\varepsilon)\Phi_{\eta,j}^{-1}\big(\widehat{N}^{\scriptscriptstyle{\! L}}_k(j-1,j)\big).$$
In any case, \eqref{martmodi} gives, for any $j\geq 1$,
\begin{eqnarray*}
W_{\varepsilon}(\widehat{N}^{\scriptscriptstyle{\! L}}_n(j,j+1))
&\lesssim& \sum_{k=0}^{n-1}
\frac{\mathbf{1}_{\{\widehat{X}^{\scriptscriptstyle{\!
L}}_k=j\textrm{ and }\widehat{X}^{\scriptscriptstyle{\! L}}_{k+1}=j-
1\}}}{w\big((1+2\varepsilon)\Phi_{\eta,j}^{-1}(\widehat{N}^{\scriptscriptstyle{\!
L}}_k(j-1,j))\big)}\\
&\lesssim& \sum_{k=0}^{\widehat{N}^{\scriptscriptstyle{\!
L}}_n(j-1,j)}
\frac{1}{w\big((1+2\varepsilon)\Phi_{\eta,j}^{-1}(k)\big)}\\
 &\lesssim& \frac{1}{1+\frac{3\varepsilon}{2}}
W(\Phi_{\eta,j+1}(\widehat{N}^{\scriptscriptstyle{\! L}}_n(j-1,j))),
\end{eqnarray*}
where we used the regular variation of $w$ for the last inequality.
Noticing also that $(1+\varepsilon)W_{\varepsilon}(x)\sim W(x)$  we
get, for $n$ large enough,
$$W(\widehat{N}^{\scriptscriptstyle{\! L}}_n(j,j+1))\le W(\Phi_{\eta,j+1}(\widehat{N}^{\scriptscriptstyle{\! L}}_n(j-1,j)))\le W(\Phi_{\eta,j+1}(\widehat{Z}^{\scriptscriptstyle{\! L}}_n(j))),$$
which concludes the proof of the lemma.
\end{proof}

\begin{proof}[Proof of Proposition \ref{modERRW}]
Before proving the  proposition, we prove a similar statement for
the reflected random walk $\widehat{X}^{\scriptscriptstyle{\! L}}$.
On the one hand, recalling that $\varepsilon$ is chosen small enough
such that $\Phi_{1/2+2\varepsilon,j_+(w)}$ is bounded, the previous
lemma insures that, for any $L$, the reflected random walk
$\widehat{X}^{\scriptscriptstyle{\! L}}$ visits site $j_+(w)$ only
finitely many time a.s. On the other hand, denoting
$\widetilde{X}^{\scriptscriptstyle{\! L}}$ the walk $\widetilde{X}$
restricted to $\lin -1,L\rin$ (reflected at $L$), it is
straightforward  that $\widetilde{X}^{\scriptscriptstyle{\! L}}\prec
\widehat{X}^{\scriptscriptstyle{\! L}}$. Copying the proof of
Proposition \ref{ERRW}, we find that, for $L\ge j_-(w)-1$,
$\widetilde{X}^{\scriptscriptstyle{\! L}}$ visits a.s. all sites of the
interval $\lin -1, j_-(w)-1\rin$ infinitely often. Thus,
according to Corollary \ref{corZinfty}, the walk
$\widehat{X}^{\scriptscriptstyle{\! L}}$ also visits a.s. all sites of
the interval $\lin-1,j_-(w)-1\rin$ infinitely often.

Now fix $L$ to be the largest integer such that the walk
$\widehat{X}^{\scriptscriptstyle{\! L}}$ satisfies
\begin{equation}\label{defJ}
\pp_\kC \{\widehat{Z}^{\scriptscriptstyle{\!
L}}_\infty(L-1)=\infty\}>0 \qquad \mbox{ and } \qquad \pp_\kC
\{\widehat{Z}^{\scriptscriptstyle{\! L}}_\infty(L)=\infty\}=0.
\end{equation}
Noticing that $\widehat{X}^{\scriptscriptstyle{\!L-1}}\prec
\widehat{X}^{\scriptscriptstyle{\! L}}$, it follows from the
previous observations that $L$ is well defined with $L\in \{j_-(w),
j_+(w)\}$ (the index $L$ can, \emph{a priori}, depend on $\kC$). We
prove that, if the initial state $\kC = (z(x),n(x,x+1))_{x\in \Z}$
satisfies
\begin{equation}\label{inikc}
z(x)\le (1+\varepsilon)n(x-1,x) \quad \mbox{ for $1\le x\le
j_+(w)$},
\end{equation}
then
\begin{equation}\label{eqpar1}
\lim_{M\rightarrow \infty}
\pp_{\mathcal{C}}\big\{\widehat{\mathcal{E}}^{\scriptscriptstyle{\!
L}}(L,M)\cap\{\forall m\ge M,\; \widehat{N}^{\scriptscriptstyle{\!
L}}_{m}(0,1)\ge \widehat{N}^{\scriptscriptstyle{\!
L}}_{m}(-1,0)\}\big\}\ge \pp_\kC
\{\widehat{Z}^{\scriptscriptstyle{\! L}}_\infty(L-1)=\infty\} \;>\;
0,
\end{equation}
where the event $\widehat{\mathcal{E}}^{\scriptscriptstyle{\!
L}}(L,M)$ is defined in the same way as $\widehat{\mathcal{E}}(L,M)$
with $\widehat{X}^{\scriptscriptstyle{\! L}}$ in place of
$\widehat{X}$. Indeed, the previous lemma yields
\begin{equation}\label{eqpar2}
\lim_{M\rightarrow \infty} \pp_{\mathcal{C}}\{ \forall m\ge M,\;
\widehat{N}^{\scriptscriptstyle{\! L}}_{m}(0,1)\ge
\widehat{N}^{\scriptscriptstyle{\! L}}_{m}(-1,0)\}=1.
\end{equation}
 Moreover, in view of \eqref{inikc}, for any $n\ge 0$ we
 have
 \begin{equation}\label{pE1}
\widehat{Z}^{\scriptscriptstyle{\! L}}_n(L)\le
(1+\varepsilon)N_n(L-1,L).
\end{equation}
Notice also that, for $j\ge 1$ and $\gamma > 1/2 + \varepsilon$,
$$\widehat{Z}^{\scriptscriptstyle{\! L}}_n(j)\;\lesssim_n\; \widehat{N}^{\scriptscriptstyle{\! L}}_n(j-1,j)+\widehat{N}^{\scriptscriptstyle{\! L}}_n(j,j+1)\;\lesssim_n\; \widehat{N}^{\scriptscriptstyle{\! L}}_n(j-1,j)+\Phi_{\gamma,j+1}(\widehat{Z}^{\scriptscriptstyle{\! L}}_n(j)),$$
where we used Lemma \ref{lemhatXK} for the upper bound. Since
$\Phi_{\gamma,j+1}(x)=o(x)$, it follows that, on the event $\{
\widehat{Z}^{\scriptscriptstyle{\! L}}_\infty(j)=\infty \}$,
\begin{equation}\label{pE2}
\widehat{Z}^{\scriptscriptstyle{\! L}}_n(j)\le
(1+\varepsilon)\widehat{N}^{\scriptscriptstyle{\! L}}_n(j-1,j)
\quad\hbox{ for $n$ large enough}.
\end{equation}
This bound can be improved for $j=1$. More precisely, for
$\gamma\in(1/2+\varepsilon,1/2+2\varepsilon)$ and $n$ large enough,
we have
\begin{eqnarray}
\nonumber\widehat{Z}^{\scriptscriptstyle{\! L}}_n(1) &\leq& \widehat{N}^{\scriptscriptstyle{\! L}}_n(0,1)+\Phi_{\gamma,2}(\widehat{Z}^{\scriptscriptstyle{\! J}}_n(1))\\
\nonumber&\leq&  \widehat{N}^{\scriptscriptstyle{\! L}}_n(0,1)+\Phi_{\gamma,2}((1+\varepsilon)\widehat{N}^{\scriptscriptstyle{\! L}}_n(0,1))\\
\nonumber&\leq& \widehat{N}^{\scriptscriptstyle{\! L}}_n(0,1)+\Phi_{1/2+2\varepsilon,2}(\widehat{N}^{\scriptscriptstyle{\! L}}_n(0,1)) \\
\label{pE3}&\leq&
\widehat{N}^{\scriptscriptstyle{\!L}}_n(0,1)+f(\widehat{N}^{\scriptscriptstyle{\!L}}_n(0,1)),
\end{eqnarray}
where we used Lemma \ref{lemmtech} for the third inequality and the
fact that $f$ satisfies (a) of Lemma \ref{WL} with $\eta = 1/2 +
2\varepsilon$ for the last inequality. Putting \eqref{defJ},
\eqref{pE1}, \eqref{pE2} and \eqref{pE3} together, we conclude that
$$\{\widehat{Z}^{\scriptscriptstyle{\! L}}_\infty(L-1)=\infty\}\subset
\bigcup_{M\geq 0}\widehat{\mathcal{E}}^{\scriptscriptstyle{\! L}}(L,M).$$
This
combined with \eqref{eqpar2}, proves \eqref{eqpar1}.

Still assuming that the initial state $\kC$ satisfies \eqref{inikc},
it follows from \eqref{eqpar1} that there exists $M$ such that
$\widehat{\mathcal{E}}^{\scriptscriptstyle{\! L}}(L,M)$ has positive
probability under $\pp_{\kC}$. On this event, the reflected walk
$\widehat{X}^{\scriptscriptstyle{\! L}}$  visits site $L$ finitely many times and thus
 $$\pp_{\mathcal{C}}\{\widehat{\mathcal{E}}^{\scriptscriptstyle{\! L}}(L,M)\cap \{\widehat{X}^{\scriptscriptstyle{\! L}} \mbox{coincides with } \widehat{X} \mbox{ forever}\}\}>0,$$
which yields
$$\pp_{\mathcal{C}}\{\widehat{\mathcal{E}}(j_+(w),M)\}\ge \pp_{\mathcal{C}}\{\widehat{\mathcal{E}}(L,M)\}>0.$$
This proves the first part of the proposition under Assumption
\eqref{inikc}. In order to treat the general case, we simply notice
that, from any initial state, the walk has a positive probability of
reaching a state satisfying \eqref{inikc}.

It remains to prove the second part of the proposition. Let $L_0$ be
the index $L$ defined in \eqref{defJ} associated with the trivial
initial state. Recalling that a state is reachable i.f.f. it can be
created from the trivial state by an excursion of a walk away from
$0$, we deduce from \eqref{eqpar1} that there exists a reachable
state $\kC$ equal to zero outside the interval $\lin -1, L_0\rin$
such that
\begin{equation}
\pp_{\mathcal{C}}\big\{\widehat{\mathcal{E}}^{\scriptscriptstyle{\!
L_0}}(L_0,0)\cap\{\forall m\ge 0,\;
\widehat{N}^{\scriptscriptstyle{\! L_0}}_{m}(0,1)\ge
\widehat{N}^{\scriptscriptstyle{\! L_0}}_{m}(-1,0)\}\big\}>0.
\end{equation}
Moreover, we have
\begin{multline*}
\lim_{n\to
\infty}\pp_{\mathcal{C}}\big\{\widehat{\mathcal{E}}^{\scriptscriptstyle{\!
L_0}}(L_0,0)\cap\{\forall m\ge 0,\;
\widehat{N}^{\scriptscriptstyle{\! L_0}}_{m}(0,1)\ge
\widehat{N}^{\scriptscriptstyle{\! L_0}}_{m}(-1,0)\}\;|\;
\widehat{\mathcal{F}}^{\scriptscriptstyle{\!
L_0}}_n\big\}\\
=\mathbf{1}_{\widehat{\mathcal{E}}^{\scriptscriptstyle{\!
L_0}}(L_0,0)\cap\{\forall m\ge 0,\;
\widehat{N}^{\scriptscriptstyle{\! L_0}}_{m}(0,1)\ge
\widehat{N}^{\scriptscriptstyle{\! L_0}}_{m}(-1,0)\}} \quad
\pp_{\kC}\hbox{-a.s.}
\end{multline*}
Hence, there exists a reachable state $\kC'=(z'(x),n'(x,x+1))_{x\in
\Z}$ equal to zero outside the interval $\lin -1, L_0\rin$ such that
\begin{equation}
\pp_{\mathcal{C}'}\big\{\widehat{\mathcal{E}}^{\scriptscriptstyle{\!
L_0}}(L_0,0)\cap\{\forall m\ge 0,\;
\widehat{N}^{\scriptscriptstyle{\! L_0}}_{m}(0,1)\ge
\widehat{N}^{\scriptscriptstyle{\!
L_0}}_{m}(-1,0)\}\big\}>\frac{3}{4}.
\end{equation}
In particular, $\kC'$ satisfies the hypotheses of the proposition.
Finally, on the event $\widehat{\mathcal{E}}^{\scriptscriptstyle{\!
L_0}}(L_0,0)$, the reflected walk
$\widehat{X}^{\scriptscriptstyle{\! L_0}}$ and $\widehat{X}$
coincide forever since they never visit site $L_0$. We conclude
that
$$\pp_{\mathcal{C}'}\{\widehat{\mathcal{E}}(j_+(w),0)\}\ge \pp_{\mathcal{C}'}\{\widehat{\mathcal{E}}(L_0,0)\}>3/4.$$
\end{proof}

\section{The walk $\bar{X}$} \label{sectionbar}
Gathering results concerning $\widetilde{X}$ and $\widehat{X}$
obtained in Sections \ref{sectiontilde} and \ref{sectionhat} we can
now describe the asymptotic behaviour of the reflected VRRW
$\bar{X}$ on the half line. The following proposition is the
counterpart of Theorem \ref{locps} for $\bar{X}$ instead of $X$.

\begin{prop}\label{locabarX} Let $\kC$ be a finite state. Under $\pp_\kC$, the following equivalences hold
\begin{equation*}
j_{\pm}(w)<\infty \quad \Longleftrightarrow \quad \bar{X} \mbox{
localizes with positive probability } \quad \Longleftrightarrow
\quad \bar{X} \mbox{ localizes a.s.}
\end{equation*}
Moreover, if the indexes $j_{\pm}(w)$ are finite, we have
\begin{eqnarray*}
&(i)&\pp_\mathcal{C}\{|\bar{R}| \le j_-(w)\}=0, \\
&(ii)&\pp_\mathcal{C}\{|\bar{R}| \le j_+(w)+1\}>0.
\end{eqnarray*}
\end{prop}

\begin{proof} The combination of  Lemma
\ref{probapositive} and Proposition \ref{modERRW} implies that, with positive $\pp_{\kC}$-probability, the walk $\bar{X}$ ultimately stays confined in the interval $\lin -1,  j_+(w)-1 \rin$. In particular, (ii) holds. Let  $j\ge
1$ be such that
\begin{equation*}
\pp_\mathcal{C}\{0<|\bar{R}| \le j\}>0.
\end{equation*}
This means that we can  find a finite state $\kC'$ such that
\begin{equation*}
\pp_{\mathcal{C}'}\big\{\{-1\} \subset \bar{R}\subset \lin -1,
j-2\rin\big\}>0.
\end{equation*}
The combination of Corollary \ref{corZinfty}, Proposition
\ref{proptildeleftbar} and Proposition \ref{ERRW} implies now that
$j\ge j_-(w)+1$. Therefore (i) holds. Furthermore, the same argument
shows that, if $j_-(w)=\infty$ then necessarily $j=\infty$ which
means that the walk does not localize. Hence, we have shown that
\begin{equation*}
 j_{\pm}(w)<\infty \quad \Longleftrightarrow \quad \bar{X} \mbox{
localizes with positive probability.}
\end{equation*}

It remains to prove that localization is, in fact, an almost sure property.
 Assume that $j_\pm(w) < \infty$ and pick  $M\ge 0$ large enough such that, starting from the trivial environment, the reflected VRRW never visits $M$ with positive probability. Given the finite state $\kC$, we choose  $x_0\ge -1$ such that all the local times of $\kC$ are zero on $\lin x_0,+\infty\irin$. Furthermore,  for $m\ge 1$, set $x_m:=M m +x_0$ and
\begin{equation*}
\tau_m:=\inf\{n\ge 0\; : \; \bar{X}_n=x_m\}.
\end{equation*}
Conditionally on $\tau_m < \infty$, the process $(\bar{X}_{\tau_m + n} - x_m)_{n\geq 0}$ is a reflected VRRW on $\lin -x_m-1,\infty\irin$ starting from a (random) finite initial state whose local times are zero for $x\geq0$. Comparing this walk with the reflected VRRW $\bar{X}$ on $\lin
-1,\infty \irin$ starting from the trivial state, it follows from Corollary \ref{corZinfty} that
$$\pp_\mathcal{C}\{\tau_{m+1}=\infty\, | \;\tau_m<\infty\}\ge \pp_0\{\bar{X} \mbox{ never visits } M\}>0,$$
which proves that $\bar{X}$ localizes a.s.
\end{proof}

 The following technical lemma will be
useful later to show that the non-reflected VRRW  localizes with
positive probability on a set of cardinality at least $2j_-(w) - 1$.

\begin{lem}\label{barXtech}
 Assume that $j_+(w)<\infty$. Then, there exists a reachable initial state $\mathcal{C}$
which is symmetric \emph{i.e.} satisfying $z(x)= z(-x)$ and
$n(x,x+1)=n(-x-x,-x)$ for all $x\ge 0$, such that
 \begin{equation*}
 \pp_{\mathcal{C}}\left\{\{\bar{R}\subset \lin -1,  j_+(w)-1 \rin\}\cap\Big\{\limsup_{n\to \infty} \frac{\bar{Z}_{\bar{\sigma}(0,n)}(1)}{\bar{Z}_{\bar{\sigma}(0,n)}(-1)}\le 1\Big\}\right\}>3/4,
 \end{equation*}
recalling the notation $\bar{\sigma}(0,n):=\inf\{k\ge 0\; :\; \bar{Z}_k(0)=n\}.$
\end{lem}

\begin{proof} Since we are dealing with the reflected random walk $\bar{X}$, the value of the state on $\ilin-\infty,-2\rin$ is irrelevant so the symmetric assumption is not really restrictive apart from the edge/site local times at $-1$ and $1$. Moreover, according to the previous proposition and the fact that $j_+(w)\le j_-(w)+1$, it follows that, on the event $\{\bar{R}\subset \lin -1,  j_+(w)-1 \rin\}$, the walk $\bar{X}$
returns to $0$ infinitely often. Hence all the hitting times $\bar{\sigma}(0,n)$ are finite. In particular, the $\limsup$ in the proposition is well-defined.

According to Proposition \ref{modERRW}, there exists a reachable state $\mathcal{C}'=(z'(x),n'(x,x+1))_{x\in \Z}$ which is  zero outside of the interval $\lin -1, j_+(w) \rin$ such that
$n'(0,1)\ge n'(-1,0)$  and for which  \eqref{c2} holds, namely
\begin{equation*}\pp_{\mathcal{C}'}\{\widehat{\mathcal{E}}(j_+(w),0)\}>3/4.
\end{equation*}
Recall that  $\widehat{\mathcal{E}}$ is the "good event" for the
modified reinforced walk $\widehat{X}$ defined by \eqref{defEKN}. On
$\widehat{\mathcal{E}}(j_+(w),0)$, by definition, we have
$\widehat{Z}_n(1)\le \widehat{N}_n(0,1)+f(\widehat{N}_n(0,1))$.
Recalling that $f(x)=o(x)$ (\emph{c.f.} (b) of Lemma \ref{WL}), we
get $\widehat{Z}_n(1)\sim \widehat{N}_n(0,1)$. Moreover, on this
event, the walk $\widehat{X}$ coincides with the reflected walk
$\widehat{X}^{\scriptscriptstyle{\! j_+(w)}}$ on $\lin -1,
j_+(w)\rin$. In particular, it follows from \eqref{eqlim-11} that
\begin{equation}\label{stablelimit}
\lim_{n \rightarrow \infty}
\frac{\widehat{Z}_n(1)}{\widehat{Z}_n(-1)}=1 \qquad
\pp_{\mathcal{C}'}\mbox{-a.s. on the event
$\widehat{\mathcal{E}}(j_+(w),0)$. }
\end{equation}
Since $\bar{X} \prec \widehat{X}$ on $\widehat{\mathcal{E}}(j_+(w),0)$, Lemma \ref{probapositive} combined with
\eqref{stablelimit} and Proposition \ref{ERRW} yield
\begin{equation}\label{eqa}
\widehat{\mathcal{E}}(j_+(w),0)\subset\left\{\{\bar{R}\subset \lin
-1, j_+(w)-1 \rin\} \cap\Big\{\limsup_{n\to \infty}
\frac{\bar{Z}_{\sigma(0,n)}(1)}{\bar{Z}_{\sigma(0,n)}(-1)}\le
1\Big\}\right\}.
\end{equation}
Consider now the reachable state $\mathcal{C}=(z(x),n(x,x+1),x\in \Z)$ obtained by symmetrizing $\kC'$ \emph{i.e.}
\begin{eqnarray*}
n(x,x+1)&=&\left\{ \begin{array}{ll}  n'(x,x+1) & \mbox{ if }x\ge 0 \\
n'(-x-1,-x) & \mbox{ if }x< 0  \\
\end{array}\right. \\
z(x)&=& n(x,x+1) + n(x-1,x).
\end{eqnarray*}
With this definition, we have $z(x)=z'(x)$ for $x\ge 1$ (recall that $\kC'$ is reachable) and since
$n'(0,1)\ge n'(-1,0)$, we also have $z(0)\ge z'(0)$ and $z(-1)\ge z'(-1)$. Now set $v(x):=z(x)-z'(x)$ for $x\geq -1$.
Defining a reflected walk $\check{X}$ on $\lin-1,\infty\irin$ with transition probabilities given for $x\geq 0$ by
$$\pp_\mathcal{C'}\{\check{X}_{n+1}=x-1\; |\; \mathcal{\check{F}}_n,\check{X}_n=x\}=
\frac{w(\check{Z}_n(x-1)+v(x-1))}{w(\check{Z}_n(x-1)+v(x-1))+w(\check{Z}_n(x+1))},
$$
it is clear that $\check{X}$ under $\pp_{\mathcal{C}'}$ has the same law as $\bar{X}$ under
$\pp_\mathcal{C}$. Besides, using  $v(-1),v(0) \geq 0$ and $v(x) = 0$ for $x\geq 1$, it follows that $\check{X}\prec \bar{X}$ under $\pp_{\mathcal{C}'}$ (just compare the transition probabilities). Using Lemma \ref{propEERW}, Corollary \ref{corZinfty} and
\eqref{eqa}, we conclude that
 \begin{multline*}
 \pp_{\mathcal{C}}\left\{\{\bar{R}\subset \lin -1,  j_+(w)-1 \rin\}\cap\Big\{\limsup_{n\to \infty} \frac{\bar{Z}_{\bar{\sigma}(0,n)}(1)}{\bar{Z}_{\bar{\sigma}(0,n)}(-1)}\le 1\Big\}\right\}\\
 \begin{aligned}
& =
 \pp_{\mathcal{C}'}\left\{\{\check{R}\subset \lin -1,  j_+(w)-1 \rin\}\cap\Big\{\limsup_{n\to \infty} \frac{\check{Z}_{\check{\sigma}(0,n)}(1)}{\check{Z}_{\check{\sigma}(0,n)}(-1)}\le
 1\Big\} \right\}\\
&\ge   \pp_{\mathcal{C}'}\left\{\{\bar{R}\subset \lin -1,  j_+(w)-1
\rin\}\cap\Big\{\limsup_{n\to \infty}
\frac{\bar{Z}_{\bar{\sigma}(0,n)}(1)}{\bar{Z}_{\bar{\sigma}(0,n)}(-1)}\le
1\Big\}\right\}\\
 &\ge
 \pp_{\mathcal{C}'}\{\widehat{\mathcal{E}}(j_+(w),0)\}>3/4.
\end{aligned}
 \end{multline*}
\end{proof}

\section{The VRRW $X$: proof of Theorem \ref{locps}}\label{sectionmaintheo}

We now have all the ingredients needed to prove Theorem \ref{locps} whose statement is rewritten below (recall that $i_\pm(w) = j_\pm(w) -1$ according to Proposition \ref{prop_iegalj}).
\begin{theo} Let $X$ be a  VRRW on $\Z$ with weight $w$ satisfying Assumption
\ref{assumw}. We have
\begin{equation}\label{eqla}
j_{\pm}(w)<\infty \Longleftrightarrow\   X \mbox{ localizes with positive
probability} \Longleftrightarrow\   X \mbox{ localizes a.s.}
\end{equation}
Moreover, when localization
occurs (\emph{i.e.} $j_\pm(w)<\infty$) we have
\begin{eqnarray}
\label{minsize}&(i)&\pp_0\{ j_-(w) <|R| < \infty \}=1\\
\label{goodsize}&(ii)&\pp_0\big\{ 2j_-(w)-1 \le |R|\le 2j_+(w)-1 \big\}>0.
\end{eqnarray}
\end{theo}
\begin{proof}
It follows directly from the definition of the VRRW and its reflected counterpart that $X \prec \bar{X}$. On the other hand, when $j_{\pm}(w)<\infty$, Proposition \ref{locabarX} states that $\bar{X}$ localizes a.s which, in  view of Corollary \ref{corZinfty}, implies $\sup_{n}X_n \leq \sup_n\bar{X}_n < \infty$ a.s. By symmetry, we conclude that $X$ localizes a.s. Reciprocally, if $X$ localizes with positive probability then there exists a finite state $\mathcal{C}$ such that
$$\pp_{\mathcal{C}}\{X \mbox{ localizes and never  visits site -1}\}>0.$$
On this event, $\bar{X}$ coincides with $X$, thus
$\pp_{\mathcal{C}}\{\bar{X} \mbox{ localizes}\}>0$. Proposition \ref{locabarX} now implies that $j_{\pm}(w)<\infty$ which concludes the proof of \eqref{eqla}.

We now prove \eqref{minsize}. Assume $j_\pm(w)<\infty$ so that $R$ is finite and not empty. Suppose by contradiction that $\pp_0\{1\le |R| \le j_-(w) \}>0$. Then, there exists a finite state $\mathcal{C}$  such that $$\pp_\mathcal{C} \{ X \mbox{ never exits the interval } \lin -1,j_-(w)-2\rin \}>0.$$
On this event, the walks $X$ and $\bar{X}$ coincide. In particular, we get $\pp_{\mathcal{C}}\{ |\bar{R}| \le j_-(w) \}>0$ which contradicts Proposition \ref{locabarX}.

It remains to establish \eqref{goodsize}. According to Lemma \ref{barXtech}, we can find a symmetric reachable initial state $\mathcal{C}$ such that
 $$\pp_{\mathcal{C}}\left\{\{\bar{R}\subset \lin -1,  j_+(w)-1 \rin\}\cap\{\limsup_{n\to \infty} \frac{\bar{Z}_{\bar{\sigma}(0,n)}(1)}{\bar{Z}_{\bar{\sigma}(0,n)}(-1)}\le 1\}\right\}>3/4.$$
Using again $X\prec \bar{X}$ together with  Proposition \ref{propEERW} and Corollary \ref{corZinfty}, we get
 $$\pp_{\mathcal{C}}\left\{\{R\subset \lin -\infty,  j_+(w)-1 \rin\}
 \cap\Big\{\{\limsup_{n\to \infty}  \frac{Z_{\sigma(0,n)}(1)}{Z_{\sigma(0,n)}(-1)}\le 1\}\cup \{Z_\infty(0)<\infty\}\Big\}\right\}>3/4.$$
The state $\mathcal{C}$ being symmetric, we also have
 $$\pp_{\mathcal{C}}\left\{\{R\subset \lin -j_+(w)+1,\infty \irin\}
 \cap\Big\{\{\limsup_{n\to \infty}  \frac{Z_{\sigma(0,n)}(-1)}{Z_{\sigma(0,n)}(1)}\le 1\}\cup \{Z_\infty(0)<\infty\}\Big\}\right\}>3/4.$$
Hence
\begin{equation}\label{eqdemi}
\pp_{\mathcal{C}}\left\{\{R\subset \lin -j_+(w)+1, j_+(w)-1  \rin\}
\cap\{\lim_{n\to \infty}
\frac{Z_{\sigma(0,n)}(-1)}{Z_{\sigma(0,n)}(1)}= 1\}\right\}>1/2,
\end{equation}
where we used that, on the event $\{R\subset \lin -j_+(w)+1, j_+(w)-1
\rin\} $, the walk $X$ visits the origin infinitely often since it cannot localize on less than
$j_-(w) + 1 \geq j_+(w)$ sites. The state $\kC$ being reachable, we already deduce that
\begin{equation*}
 \pp_0\{  1\le |R| \le 2j_+(w)-1 \}>0.
\end{equation*}
Next, for $\gamma\in (0,1/2)$, define
\begin{equation*}
\mathcal{G}_\gamma:=\Big\{R\subset \lin -j_+(w)+1, j_+(w)-1
\rin\Big\}
 \cap\Big\{\forall n\ge 0,\; \gamma\le\frac{Z_{\sigma(0,n)}(1)}{ Z_{\sigma(0,n)}(-1)+Z_{\sigma(0,n)}(1)}\le 1-\gamma\Big
 \}.
 \end{equation*}
According to \eqref{eqdemi}, for any given $\gamma$,
 there exists a reachable configuration $\mathcal{C}'$ such that
$\pp_{\mathcal{C}'}\{\mathcal{G}_\gamma\}>0$. Thus, it suffices to prove that, for $\gamma$ close enough to $1/2$, we have
\begin{equation}\label{eqfin}
\mathcal{G}_\gamma\subset \{ 2j_-(w)-1\le |R| \le 2j_+(w)-1\}\quad\hbox{$\pp_{\kC'}$-a.s.}
 \end{equation}
To this end, we introduce the walk $\breve{X}$ on $\lin 0,\infty \irin$ with the same transition
probabilities as the walk $\widetilde{X}$  studied in Section
\ref{sectiontilde} except at site $x=0$ where we define
$$\pp_{\kC'}\{\breve{X}_{n+1}=1 \;| \;\breve{\mathcal{F}}_n, \breve{X}_n=0\}=1-\pp_{\kC'}\{\breve{X}_{n+1}=0 \;|\;
\breve{\mathcal{F}}_n, \breve{X}_n=0\}=\gamma$$ (\emph{i.e.} when
this walk visits 0, it has a positive probability of staying at the origin at the next step). Using exactly the same arguments as in Proposition \ref{ERRW}, we see that $\breve{X}$ localizes a.s. under $\pp_{\kC'}$ and that the bounds \eqref{hyprec} obtained for $\widetilde{X}$ give similar estimates for $\breve{X}$: for
 $j\ge 1$  and  $\alpha\in (0,\gamma)$,
\begin{equation*}\Phi_{\gamma-\alpha,j}(\breve{Z}_k(j-1))
\lesssim \breve{N}_k(j-1,j)
 \lesssim \Phi_{\gamma+\alpha,j}(\breve{Z}_k(j-1)) \qquad\pp_{\kC'}\mbox{-a.s.}
 \end{equation*}
Thus, we can now choose $\gamma$ close enough to $1/2$ such that, $j_{\gamma-\alpha}(w)=j_-(w)$ for some $\alpha>0$.
The previous estimate implies, by induction, that the localization set of $\breve{X}$ is such that
\begin{equation}\label{eqbreve}
\lin 0,j_-(w)-1\rin \subset  \breve{R}\qquad \pp_{\kC'}\mbox{-a.s.}
\end{equation}
Finally, consider the walk $X^+$ on $\lin 0,\infty \irin$ obtained from $X$ by keeping only its excursions on the half-line $\lin 0,+\infty\rin$ \emph{i.e.}
\begin{equation*}
X^+_n := X_{\zeta_n},
\end{equation*}
where $\zeta_0 := 0$ and $\zeta_{n+1} := \inf\{k>\zeta_n\; : \; X_k\geq 0\}$. On the event
$\mathcal{G}_\gamma$, the r.v. $\zeta_n$ are finite. Recalling the construction described in Section \ref{seccouplage} of the VRRW $X$ from a sequence $(U_i^x,x \in \Z,i\ge 1)$ of i.i.d. uniform
random variables, we see that, on $\mathcal{G}_\gamma$  we  have
$$
U_n^0\ge 1-\gamma \quad\Longrightarrow\quad X^+_{\sigma^+(0,n)+1}=X^+_{\sigma^+(0,n)}+1=1\qquad\hbox{(for $n$ larger than the initial local
time at $0$)}.
$$
We also construct $\breve{X}$ from the same random variables
$(U_i^x)$ (the walk is not nearest neighbour at $0$  so we set
$\breve{X}_{\breve{\sigma}(0,n)+1}=1$ if $U_n^0\ge 1-\gamma$ and
$\breve{X}_{\breve{\sigma}(0,n)+1}=0$ otherwise). Then, it follows
from the previous remark that $\breve{X}\prec X^+$ on
$\mathcal{G}_\gamma$. Using one last time Corollary \ref{corZinfty}
and \eqref{eqbreve}, we deduce that
$$\mathcal{G}_\gamma \subset  \{X \mbox{ visits $j_-(w)-1$ i.o.}\}\qquad\hbox{$\pp_{\kC'}$-a.s.}$$
By invariance of the event $\mathcal{G}_\gamma$ under the space reversal $x \mapsto -x$, we conclude that
$$\mathcal{G}_\gamma\subset \{X \mbox{ visits $j_-(w)-1$ and $-(j_-(w)-1)$ i.o.}\}\qquad\hbox{$\pp_{\kC'}$-a.s.}$$
hence \eqref{eqfin} holds.
\end{proof}

\section{Asymptotic local time profile}\label{sectionlocaltime}

Although Theorem \ref{locps} is only concerned with the size of the
localization set, looking back at the proof, we see that we can also
describe, with little additional work, an asymptotic local time
profile of the walk (but we cannot prove that other asymptotics do
not happen). Let us give a rough idea of how to proceed while
leaving out the cumbersome details. In order to simplify the
discussion, assume that $i_\pm(w)$ are finite and that both indexes
are equal. Hence, the VRRW $X$ localizes with positive probability
on the interval $\lin -i_\pm(w),i_\pm(w)\rin$. Looking at the proof
of \eqref{goodsize}, we see that, with positive probability, the urn
at the center of the interval is balanced, \emph{i.e.}
\begin{equation}\label{stabcenter}
Z_n(-1)\sim Z_n(1)\sim \frac{Z_n(0)}{2}.
\end{equation}
This tells us that, with positive probability, as $n$ tends to
infinity, the local times $Z_n(1), Z_n(2),\ldots,Z_n(i_\pm)$ and
$Z_n(-1), Z_n(-2),\ldots,Z_n(-i_\pm)$ are of the same magnitude as
$\bar{Z}_n(1),\bar{Z}_n(2),\ldots,\bar{Z}_n(i_\pm)$ for the
reflected random walk $\bar{X}$ on $\lin-1,\infty\irin$.
Furthermore, recalling that $\widetilde{X} \prec \bar{X} \prec
\widehat{X}$ on $\widehat{\mathcal{E}}(i_\pm(w)+1,0)$, we can use
\eqref{hyprec} and \eqref{hyprec2} to estimate the local times of
$\bar{X}$, which therefore also provides asymptotic for the local
times of the non-reflected walk $X$. More precisely, given  a family
of functions $(\chi_\eta(x),\eta \in (0,1))$, introduce the notation
$$f(x)\asymp \chi_{\eta_0}(x) \quad\mbox{ if  }\quad
\chi_{\eta_0-\varepsilon}(x)\ \le \ f(x)\ \le \
\chi_{\eta_0+\varepsilon}(x)\quad \mbox{ for all $\varepsilon>0$ and
$x$ large enough.} $$
 Then, one can prove that, with positive probability, the VRRW localizes on $\lin
-i_\pm(w),i_\pm(w)\rin$ in such a way that \eqref{stabcenter} holds
and that, for every $1\leq i\leq i_\pm(w)$,
 $$\left\{\begin{array}{c} Z_n(i)\asymp \Phi_{1/2,i}(Z_n(i-1))\\
  Z_n(-i)\asymp \Phi_{1/2,i}(Z_n(-i+1)^{\vphantom{X^{X^X}}})\end{array}\right.\quad\hbox{ as $n$ goes to
infinity,}$$ where $(\Phi_{\eta,i}, \eta\in(0,1))$ is the family of
functions defined in \eqref{defphi}. Recalling that
$\Phi_{\eta,i}(x)=o(x)$ for any $i\ge 2$, we deduce in particular
$$Z_n(-1)\sim Z_n(1)\sim \frac{Z_n(0)}{2}\sim \frac{n}{4},$$
\emph{i.e.} the walk spends almost all its time on the three center
sites $\{-1,0,1\}$. Furthermore, setting
\begin{equation}
\Psi_{\eta,i}(x):= \Phi_{1/2,i}\circ \Phi_{1/2,i-1} \circ \ldots
\circ \Phi_{\eta,1}(x/2),
\end{equation}
 we
 get, for any $i\in \lin 1, i_\pm(w) \rin$
\begin{equation}\label{tl}\left\{\begin{array}{c} Z_{n}(i)\asymp \Psi_{1/2,i}(n)\\
  Z_{n}(-i)\asymp\Psi_{1/2,i}(n)^{\vphantom{X^{X^X}}}\end{array}\right.\quad\hbox{ as $n$ goes to
infinity,}
\end{equation}
(\emph{c.f.} Figure \ref{fig1}). The calculation of this family of
functions may be carried out explicitly in some cases. For example,
if we consider a weight sequence of the form $w(k)\sim k
\exp(-\log^\alpha k)$ for some $\alpha \in (0,1)$, then, with
arguments similar as those used in the proof of Proposition
\ref{calculi}, we can estimate the functions $\Psi_{\eta,i}(n)$ and,
after a few  lines of calculus, we conclude that, in this case,
$$Z_n(i) \; = \; \frac{n}{\exp( (\log n)^{(1-\alpha)(i-1)+ o(1)})}\qquad \mbox{for  $i\in \lin 1, i_\pm(w) \rin$}.$$

\begin{figure}[!top]
\begin{center}
\includegraphics[width=14cm]{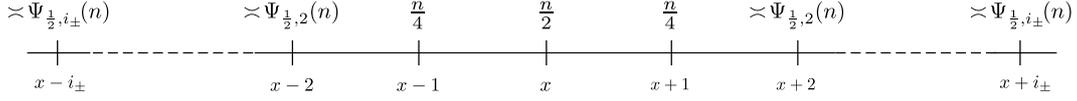}
\end{center}
\caption{Local time profile at time $n$.} \label{fig1}
\end{figure}

\section{Appendix: proof of Proposition \ref{theoic3}}\label{sectionProp}
The proof of Proposition \ref{theoic3} is largely independent of the
rest of the paper and uses arguments similar to those developed in
\cite{T1,T2} and then in \cite{BSS}. First, let us remark that the first part of the
proposition is a direct consequence of Theorem $1.1$ of \cite{BSS}.
Thus, we just prove (ii). Assume that localization on $5$ sites
occurs with positive probability and let us prove that necessarily
$i_-(w) = 2$. From now on, let $\bar{X}$ denote the VRRW restricted
to $\lin 0, 4 \rin$ (\emph{i.e.} reflected at sites $0$ and $4$).
Then, Lemma 3.7 of \cite{BSS} insures that there exists some initial
state $\mathcal{C}$ such that $\pp_{\mathcal{C}}\{ \mathcal{H}
\}>0$, where the event $\mathcal{H}$ is defined by
$$\mathcal{H}:=\{\lim_{n\to \infty}\bar{Y}_n^+(0)<\infty\}\cap \{\lim_{n\to \infty}\bar{Y}_n^-(4)<\infty\} $$
with
\begin{equation*}\label{defY}
\bar{Y}_n^{\pm}(x) := \sum_{k=0}^{n-1}
\frac{\mathbf{1}_{\{\bar{X}_k=x\textrm{ and }\bar{X}_{k+1}=x\pm
1\}}}{w(\bar{Z}_k(x\pm 1))}\qquad \mbox{ for $x\in \Z$}.
\end{equation*}
Setting $\bar{M}_n(x) := \bar{Y}_n^+(x)-\bar{Y}_n^-(x)$, we have,
for any $x$,
 \begin{equation}\label{eqWb}
W(\bar{Z}_{n}(x+2))-W(\bar{Z}_{n}(x))=\bar{Y}_n^-(x+3)-\bar{Y}_n^+(x-1)+\bar{M}_{n}(x+1)+
C(x),
\end{equation}
where $C(x)$ is some constant depending only on $x$ and the initial
state $\kC$. Moreover,  for $x\in  \lin 1, 3 \rin$,  the process
$(\bar{M}_n(x),n\geq 0)$ is a martingale bounded in $L^2$.
 Therefore, recalling the notation $\equiv$ defined in the beginning of Section \ref{sectiontilde}, the a.s. convergence of
 $\bar{M}_n(2)$ gives
 $$W(\bar{Z}_n(3))\equiv W(\bar{Z}_n(1))\qquad \mbox{ on  $\mathcal{H}$}.$$
Using Lemma \ref{limborne} and the fact that
$\bar{Z}_n(3)+\bar{Z}_n(1)\sim n/2$, we deduce that
$$\bar{Z}_n(1)\sim \bar{Z}_n(3)\sim \frac{n}{4} \qquad \mbox{ on $\mathcal{H}$}.$$
Besides, the convergence of the martingale  $\bar{M}_n(3)$ combined
with the fact that $\bar{X}$ is reflected at site 4 imply that
$$ \bar{Y}_n^-(3)\equiv \bar{Y}_n^+(3)\equiv W(\bar{Z}_n(4)).$$
Hence, taking $x=0$ in \eqref{eqWb}, we get
\begin{equation}\label{eqWb2}
W(\bar{Z}_n(2))\equiv W(\bar{Z}_n(0))+W(\bar{Z}_n(4)).
\end{equation}
Define $I_n :=\min(\bar{Z}_n(0),\bar{Z}_n(4))$ and $S_n
:=\max(\bar{Z}_n(0),\bar{Z}_n(4))$. The previous equation gives
$$\limsup_{n\to \infty} \frac{W(I_n)}{W(\bar{Z}_n(2))}\le \frac{1}{2} \quad\mbox{and}\quad \limsup_{n\to \infty} \frac{W(S_n)}{W(\bar{Z}_n(2))}\le 1,$$
which implies, in view of Lemma \ref{limborne},
$$\limsup_{n\to \infty}\, \frac{I_n}{\bar{Z}_n(2)}=0 \quad\mbox{and}\quad \limsup_{n\to \infty}\, \frac{S_n}{\bar{Z}_n(2)}\le 1.$$
Using that $I_n+S_n+\bar{Z}_n(2)\sim n/2$, we get
$$\liminf_{n\to \infty}\, \frac{\bar{Z}_n(2)}{n/4}\ge 1.$$
In particular, denoting $K_n:=\max(\bar{Z}_n(1),\bar{Z}_n(3))\sim
n/4$, we deduce that for any $\delta>0$ and for $n$ large enough,
$$\bar{Z}_n(2)\ge (1-\delta)K_n.$$
On the other hand, Equation \eqref{eqWb2} shows that there exists a
(random) constant $\gamma$, such that for $n$ large enough,
$$W(\bar{Z}_n(2))\le 2W(S_n)+\gamma.$$
Hence, we find that
$$K_n\le \frac{1}{1-\delta}W^{-1}\left(2W(S_n)+\gamma\right).$$
Therefore, we have
\begin{multline*}
\bar{Y}^+_\infty(0)+\bar{Y}_\infty^-(4) \, = \,
\sum_{n=0}^\infty\frac{\mathbf{1}_{\{\bar{X}_n=0\}}}{w(\bar{Z}_n(1))}+\frac{\mathbf{1}_{\{\bar{X}_n=4\}}}{w(\bar{Z}_n(3))}
\,\ge\, \sum_{n=0}^\infty
\frac{\mathbf{1}_{\{\bar{X}_n\in\{0,4\}\}}}{w(K_n)}
 \,\ge\, c\sum_{n=0}^\infty
 \frac{\mathbf{1}_{\{\bar{X}_n\in\{0,4\}\}}}{w\left(\frac{1}{1-\delta}W^{-1}\left(2W(S_n)+\gamma\right)\right)}\\
 \,\ge\,  c'\sum_{k=0}^\infty
\frac{1}{w\left(W^{-1}\left(2W(k)+\gamma\right)\right)},
\end{multline*}
for some constants $c,c' > 0$. Recalling that
$$\Phi_{\eta,3}(x)=W^{-1}\left(\int_0^x \frac{dt}{w(\eta W^{-1}(W(x)/\eta))}\right),$$
we deduce that if $Y^+_\infty(0)+Y_\infty^-(4)$ is finite with
positive probability, then $\Phi_{\eta,3}(x)$ is bounded for any
$\eta<1/2$. This means that $i_-(w)=2$, which concludes the proof of
the proposition.


\begin{thebibliography}{99}

\bibitem{BSS} \textbf{Basdevant A.-L., Schapira B., Singh A.} Localization on $4$ sites for Vertex Reinforced Random Walk on $\Z$. Preprint, arXiv:1201.0658.

\bibitem{BGT} \textbf{Bingham N., Goldie C., Teugels, J.} Regular Variation. \textit{Encyclopedia of Mathematics and its Applications, 27. Cambridge University Press, Cambridge,
1989.}

\bibitem{Dav} \textbf{Davis B.} Reinforced random walk. \textit{Probab. Theory Related Fields 84, (1990),
203--229.}

\bibitem{ETW1} \textbf{Erschler A., T\'oth B., Werner W.} Some locally self-interacting walks on the integers. \textit{Probability in Complex Physical Systems, (2012),
313--338.}

\bibitem{ETW2} \textbf{Erschler A., T\'oth B., Werner W.} Stuck Walks. \textit{To appear in Probab. Theory Related
Fields.}




\bibitem{PE} \textbf{Pemantle R.} Phase transition in reinforced random walk and RWRE on trees. \textit{Ann. Probab. 16, (1988), no. 3, 1229--1241.}


\bibitem{P} \textbf{Pemantle R.} Vertex-reinforced random walk. \textit{Probab. Theory Related Fields 92, (1992),
117--136.}

\bibitem{PS}  \textbf{Pemantle R.} A survey of random processes with reinforcement. \textit{Probab. Surv. 4, (2007),
1--79.}

\bibitem{PV} \textbf{Pemantle R., Volkov S.} Vertex-reinforced random walk on $\Z$ has finite range. \textit{Ann. Probab. 27, (1999),
1368--1388.}

\bibitem{Sch} \textbf{Schapira B.} A $0-1$ law for Vertex Reinforced Random Walk on $\Z$ with weight of order $k^\alpha$, $\alpha<1/2$. \textit{Electron. Comm. Probab.  17 no. 22 (2012),
1--8.}

\bibitem{T1} \textbf{Tarr\`es P.} Vertex-reinforced random walk on $\Z$ eventually gets stuck on five points. \textit{Ann. Probab. 32, (2004),
2650--2701.}

\bibitem{T2} \textbf{Tarr\`es P.} Localization of reinforced random walks. Preprint, arXiv:1103.5536.


\bibitem{V2} \textbf{Volkov S.} Phase transition in vertex-reinforced random walks on $\Z$ with non-linear reinforcement. \textit{J. Theoret. Probab. 19, (2006),
691--700.}


\end{thebibliography}
\end{document}